\documentclass{amsart}

\numberwithin{equation}{section}
\usepackage{amssymb, amsmath}
\usepackage{enumerate, xspace}
\usepackage[normalem]{ulem}
\usepackage{marvosym} 
\usepackage{bbm}       
\usepackage{changebar}
\usepackage[backgroundcolor=blue!20!white, linecolor=blue!20!white,
textsize=footnotesize]{todonotes}
\usepackage[colorlinks=true,linkcolor=blue,citecolor=magenta]{hyperref}
\usepackage{graphicx}
\usepackage[T1]{fontenc}
\usepackage[utf8]{inputenc}
\usepackage{bbm}

\newtheorem{thm}{Theorem}[section]
\newtheorem{lem}[thm]{Lemma}
\newtheorem{sub-lem}[thm]{Sub-Lemma}

\newtheorem{sublem}[thm]{Sub-lemma}
\newtheorem{prop}[thm]{Proposition}

\newtheorem{defin}[thm]{Definition}

\newtheorem{rem}[thm]{Remark}

\newcommand\cA{{\mathcal A}}

\newcommand\cC{{\mathcal C}}

\newcommand\cP{{\mathcal P}}

\newcommand\cW{{\mathcal W}}
\newcommand\cZ{{\mathcal Z}}

\newcommand\bC{{\mathbb C}}

\newcommand\bN{{\mathbb N}}

\newcommand\bR{{\mathbb R}}

\newcommand\bW{{\mathbb W}}
\newcommand\bZ{{\mathbb Z}}

\newcommand\bbf{{\mathbbm f}}

\newcommand\Id{{\mathbbm{1}}}
\newcommand\ve{\varepsilon}

\newcommand\vf{\varphi}

\newcommand\bsP{{\boldsymbol P}}

\newcommand{\Const}{C_{\#}}

\newcommand{\norm}[1]{\left\|#1\right\|}
\newcommand{\abs}[1]{\left|#1\right|}

\newcommand {\notcurlywedge}{\not\hskip-3pt\curlywedge\hskip2pt}
\begin{document}

\title{Piecewise contractions}
\author{Sakshi Jain}
\address{Sakshi Jain\\
Monash University\\
9 Rainforest Walk, Melbourne 3800, Australia.}
\email{{\tt sakshi.jain@monash.edu}}
\author{Carlangelo Liverani}
\address{Carlangelo Liverani\\
Dipartimento di Matematica\\
II Universit\`{a} di Roma (Tor Vergata)\\
Via della Ricerca Scientifica, 00133 Roma, Italy.}
\email{{\tt liverani@mat.uniroma2.it}}
\date{\today}
\begin{abstract}
We study piecewise injective, but not necessarily globally injective, contracting maps on a compact subset of \(\bR^d\). We prove that generically the attractor and the set of discontinuities of such a map are disjoint, and hence the attractor consists of periodic orbits. In addition, we prove that piecewise injective contractions are generically topologically stable.
\end{abstract}
\keywords{Contracting maps, attractors, generic properties.}
\thanks{We thank Henk Bruin, Oliver Butterley, Giovanni Canestrari, Roberto Castorrini, Stefano Galatolo, Arnaldo Nogueira, Benito Pires, and Antonio Rapagnetta for helpful discussions, comments, and suggestions. We would also like to thank the anonymous referee for pointing out the inconsistencies and typographical errors in the previous version. This work was supported by the PRIN Grants ``Regular and stochastic behaviour in dynamical systems" (PRIN 2017S35EHN), ``Stochastic properties of dynamical systems" (PRIN 2022NTKXCX), and by the MIUR Excellence Department Project Math@TOV awarded to the Department of Mathematics, University of Rome Tor Vergata. SJ was additionally funded by ARC (DP220100492) and acknowledges Monash University and CL acknowledges membership to the GNFM/INDAM}
\maketitle

\tableofcontents

\section{Introduction}
     Studying the dynamical properties of discontinuous hyperbolic dynamical systems is important to understand many relevant systems (like billiards and optimal control theory etc.) but it is difficult if the system is multi-dimensional. A first step to address this problem can be to divide the problem into two cases: the piecewise expanding case and the piecewise contracting case. Here we address the case of piecewise contracting systems. Since the attractor can differ drastically depending on the kind of piecewise contraction, a first approach is to identify a large class of piecewise contractions that exhibit similar dynamical properties.  \par
     In recent decades, the properties of the attractor of piecewise contractions have been studied under different settings  (see \cite{BD}, \cite{CGM}, \cite{CGMU}, \cite{NP}, \cite{NPR}). To begin with, H. Bruin, J.H.B. Deane (see \cite{BD}) studied families of piecewise linear contractions on the complex plane \(\bC\) and proved that, for almost all parameters, each orbit is asymptotically periodic. \par
	 In the case of {\em one dimensional} piecewise contractions, A. Nogueira and B. Pires, in \cite{NP}, studied {\em globally} injective piecewise contractions on half closed unit interval \([0,1)\) with partition of continuity consisting of \(n\) elements and concluded that such maps have at most \(n\) periodic orbits, that is, the attractor can be a Cantor set or a collection of at most \(n\) periodic orbits or a union of a Cantor set and at most \(n-1\) periodic orbits. In particular, when the attractor consists of exactly \(n\) periodic orbits, the map is asymptotically periodic.\footnote{The limit set of every element in the domain is a periodic orbit.} They also proved that every such map on \(n\) intervals is topologically conjugate to a piecewise linear contraction of \(n\) intervals whose slope in absolute value equals 1/2. We would like to emphasise that this result does not imply that any two piecewise contractions close enough to each other are topologically conjugate to each other. Hence, this result is not a stability result for the class of piecewise contractions. In \cite{GN}, piecewise increasing contractions on \(n\) intervals are considered and it is proved that the maximum umber of periodic orbits is \(n\). They also prove that the collection parameters which give a piecewise contraction with non-asymptotically periodic orbits is a Lebesgue null measure set whose Hausdorff dimension is large or equal to \(n\). In \cite{NPR}, A. Nogueira, B. Pires  and  R. Rosales proved that generically (under \(\cC^0\) topology) globally injective piecewise contractions of \(n\) intervals are asymptotically periodic and have at least one and at most \(n\) internal\footnote{Such orbits persist under a sufficiently small \(\cC^0\)-perturbation, refer to \cite{NPR} for precise definition.} periodic orbits. In \cite{CCG}, A. Calderon, E. Catsigeras and P. Guiraud proved that the attractor of a piecewise injective contraction consists of finitely many periodic orbits and minimal Cantor sets. In \cite{LN18,JO}, they study symbolic coding associated to piecewise contractions on the unit interval proving that they are rleated to the symbolic coding  of rotations of the circle.\par
	 In {\em higher dimensions}, E. Catsigeras and R. Budelli (see \cite{CB}) proved that a finite dimensional piecewise contracting system with separation property,\footnote{Injective on the entire domain except for the discontinuity set.} generically (under a topology that is finer than the topology we use for proving openness and coarser than the topology we use for proving density) exhibits at most a finite number of periodic orbits as its attractor. Here, we obtain similar, actually stronger, results without assuming the separation property.\par
     In any (finite) dimension, in \cite{CGMU}, the authors show that if the set of discontinuities and the attractor of a piecewise contraction are mutually disjoint, then the attractor consists of finitely many periodic orbits. This result is a byproduct of our arguments as well. In \cite{GT88} they study symbolic dynamics associated to piecewise contractions (referred to as quasi-contractions in that article) and categorise its association into different kinds of circle rotations.\par
	  After having sumitted this article for submission, we noticed a new submission \cite {GP24} the next month, where they provide a measure-theoretical criteria for asymptotic periodicity of a parametrised family of locally bi-Lipschitz piecewise contractions on a compact metric space.\par
	 Nevertheless, the occurrence of chaotic behavior in such systems has been addressed in the literature,  \cite{KR} provides an example of a piecewise affine contracting map with positive entropy. The presence of a Cantor set in the attractor has also been studied rigorously. Examples of such maps, for one-dimension, are given in \cite[Example 4.3]{CGMU}, \cite{RC}  and, for three-dimensions, in \cite{CGMU} (Example 5.1) where it is also proved that such a piecewise contraction turns out to have positive topological entropy. In \cite{Pires19}, it is proven that, given a minimal interval exchange transformation with any number of discontinuities, there exists an injective piecewise contractions with Cantor attractors topologically conjugate to it, also conversely that, piecewise contractions with Cantor attractors are topologically semiconjugate to topologically transitive interval exchange transformations. Additionally, in \cite{Pires19} (respectively in \cite{CGM}), it is proved that the complexity\footnote{The complexity function of the itineraries of orbits, refer to \cite{CGM}, \cite{Pires19}.} of a {\em globally injective} piecewise contractions (respectively, piecewise contraction with separation property) on the interval is eventually affine (which is eventually constant in the case of piecewise contractions with periodic attractors).  \par
	 The global picture presented by the above articles is that the Cantor attractors are rare, but can exist in exceptional (but constructible) cases, and many such explicit examples have been rigorously studied. \par
     Piecewise contractions have also been used to study some models of outer billiards (see \cite{Gaivão20}, \cite{Jeong12}, \cite{MGG}) where a billiard map is constructed such that it is a piecewise contraction, and so the properties of piecewise contractions are relevant in the study of such billiard maps.\par
	 Note that in the above papers (and in this article), maps with only finitely many partition elements are considered.     \par
	 The layout of this article is as follows:\\
     Section~\ref{sec:results} is dedicated to definitions, settings, and the statement of results. In Section~\ref{sec:op} we prove that the set of piecewise contractions with attractor disjoint from the set of discontinuities is open, under a rather coarse topology, and that if the maps are also piecewise injective (hence not necessarily globally injective), then they are topologically stable.
     In Section~\ref{sec:den}, we prove that piecewise contractions with the attractor disjoint from the set of discontinuities are dense, under a rather fine topology, among the piecewise injective smooth contractions. Finally, we have three appendices collecting some needed technical facts.
\section{Settings and results} \label{sec:results}
      Throughout this article, we work with \((X,d_0)\subset (\bR^d,d_0)\), a  compact subset of \(\bR^d\) where \(d\in\bN\) and \(d_0\) is the standard euclidean metric on \(\bR^d\). Under these settings, we define a piecewise contraction as follows:
\begin{defin}[Piecewise contraction]\label{def:pc}
     A map \(f:X\to X\) is called a piecewise contraction if $\overline{f(X)}\subset  \mathring{X}$ and there exist \(m\in \bN\) and a collection \(\bsP(f)=\{P_i: P_i=\mathring{ P_i}\}_{i=1}^m\) of subsets  of \(X\) such that : 
\begin{itemize}
	 \item \(X=\cup_{i=1}^{m}\overline P_i\) where \( P_i\cap P_j=\emptyset\) whenever \(i\neq j\).
	 \item \(f|_{P_i}\) is a uniform contraction, that is, there exists \(\lambda \in (0,1)\) such that for all \(i\in\{1,2,\ldots,m\}\), and for any \(x,y\in P_i\),
	     \[
		 d_0(f(x),f(y))\leq \lambda d_0(x,y).
		 \]
	 \item  there exists a partition $\{\tilde P_i\}$, $\tilde P_i\cap \tilde P_j=\emptyset$ for $i\neq j$, $P_i\subset \tilde P_i\subset \overline{P}_i$, \(X=\cup_{i=1}^{m}\tilde P_i\) such that $f|_{\tilde P_i}$ is continuous.
\end{itemize} 
     Here \(\lambda\) is called the `contraction coefficient' of \(f\), \(\bsP(f)\) is called the `partition of continuity,' and \(\overline{P_i}\) and \(\partial P_i\) represent the closure and boundary respectively of a partition element \(P_i\).
\end{defin}

\begin{rem} The third condition pertaining to the partition $\tilde P$, states that the values of $f$ on the boundary of partition elements must be the limit of the values of $f$ inside some elements, but no particular condition is imposed on which element.
\end{rem}
     For a piecewise contraction \(f\) with partition of continuity \(\bsP(f)=\{P_1, P_2, \ldots,P_m\}\), define \(\partial\bsP(f)\) as the boundary of the partition \(\bsP(f)\) given by 
         \[
         \partial\bsP(f)=\cup_{i=1}^m \partial P_i.
         \] 
We denote by \(\Delta(f)\), the union of the set of discontinuities of \(f\) with \(\partial X\). To avoid confusion in the choice of partition of continuity for a given piecewise contraction, we define the following:
\begin{defin}[Maximal Partition]
     A partition \(\bsP(f)\) is called the maximal partition of a piecewise contraction \(f\) if \(\partial\bsP(f)=\Delta(f)\).
\end{defin} 
     
Any partition of continuity of \(f\) is a refinement of the maximal partition. For a piecewise contraction with maximal partition \(\bsP(f)=\{P_1^1,P_2^1,\ldots,P_m^1\}\), let us define the  partition  \(\bsP(f^n)= \{P_1^n,P_2^n,\ldots,P_{m_n}^n\}\), relative to the \(n\)th iterate \(f^n\) of \(f\),  where, for every \(k\in\{1,2,\ldots,m_n\}\),
\begin{equation}{\label{eq:n-partn}} 
	 P_{k}^n= P_{i_0}^1\cap f^{-1}(P_{i_1}^1)\cap f^{-2}(P_{i_2}^1)\cap \ldots\cap f^{-(n-1)}(P_{i_{n-1}}^1). 
\end{equation} 
     where \(i_j\in\{1,2,\ldots,m\}\) for every \(j\in\{0,1,\ldots,n-1\}\). Note that \(\bsP(f)\) being the maximal partition of \(f\) does not imply that the partition \(\bsP(f^n)\) is the maximal partition of \(f^n\).
\begin{rem}
	  Throughout this article, for a given piecewise contraction \(f\), the partition we consider is the maximal partition \(\bsP(f)\), whereas for its iterates \(f^n\) we consider the partition \(\bsP(f^n)\), as defined in equation~(\ref{eq:n-partn}), which may not be the maximal partition of \(f^n\).  
\end{rem}
 One of the goals of this article is to understand the attractor of piecewise contractions. On that note we recall the standard definition of attractor:
\begin{defin}[Attractor]\label{def:att}
	 For a piecewise contraction \(f\) with \(\bsP(f)=\{P_i\}_{i=1}^m\), the attractor is defined as \(\Lambda(f)=\bigcap\limits_{n\in \bN}\overline{f^n(X)}\).
\end{defin}

When working with discontinuous maps, it is natural to talk of Markov maps (maps with Markov partitions), so we first give the following definition followed by the definiton of Markov maps in our settings:

\begin{defin}[Stabilisation of Partition]\label{def:stp}
	For a piecewise contraction \(f\), we say that the maximal partition  \(\bsP(f)\) stabilises if there exists \(N\in \bN\) such that for all \(P\in \bsP(f^N)\) there exists \(Q\in \bsP(f^N)\) such that \(\overline{f^N(P)}\subset Q\). We call the least such number \(N\), the ``stabilisation time" of \(\bsP(f)\). 
\end{defin}
 
\begin{defin}[Markov Map]\label{def:markov}
	 A piecewise contraction whose maximal partition stabilises is called a Markov map.
\end{defin}	

Now, we are able to state our first result (proved in Section~\ref{sec:op}):

\begin{thm}\label{thm:eq}
	A piecewise contraction \(f\) with \(\Lambda(f)\) as the attractor and \(\Delta(f)\) as the union of the set of discontinuities and \(\partial X\), satisfies that \(\Lambda(f)\cap\Delta(f)=\emptyset \)  if and only if it is Markov. Moreover, the attractor of a Markov map consists of periodic orbits.
   \end{thm}
   
\begin{rem}\label{rem:symbolic_dyn}
Note that, given a Markov map $f$, $N$ its stabilisation time and $\bsP(f^N)=\{Q_1,\dots, Q_l\}$ the associated partition, then $f$ induces a dynamics $\bbf$ on $\Omega(f):=\{1,\dots, l\}$ by the rule $f(Q_i)\subset Q_{\bbf(i)}$. See Lemma~\ref{lem:refined_partition}.
\end{rem}
To state further results, we need to add some hypotheses on our system, and thus we give the following definition:
\begin{defin}[Piecewise injective contraction]\label{def:pic}
	A piecewise contraction \(f\) with partition \(\bsP(f)=\{P_i\}, i\in\{1,2,\ldots,m\}\), is called a \emph{piecewise injective contraction} if, for all \(i\in\{1,2,\ldots,m\}\), there exists $U_i=\mathring U_i\supset  \overline{P_i}$ and an injective contraction $\tilde f_i: U_i\to\bR^d$ such that \(\tilde f_i|_{U_i}\supset P_i\) and \(\tilde f_i|_{P_i}=f|_{P_i}\).
\end{defin}
For any piecewise contractions \(f,g \) with partitions \(\bsP(f)=\{P_1,P_2,\ldots,P_m\}\), and \(\bsP(g)=\{Q_1,Q_2,\ldots,Q_m\}\) respectively, define 
\[
\begin{split}
H(\bsP(f),\bsP(g))=\Big\{\psi\in \cC^0(X,X)\;:\;& \psi\text{  is an homeomorphism} \\
& \forall\ P \in {\bsP}(f),\ \exists\, !\   Q \in {\bsP}(g): \psi(P)=Q \Big\}.
\end{split}
\] 
We define the distance \(\rho\) as follows (see Lemma~\ref{lem:rho} for the proof that \(\rho\) is a metric):
\[
     \rho(f,g):=
	 \begin{cases}
	   \quad \quad \quad \quad \quad \quad A &\hskip-1cm\text{if }   H(\bsP(f),\bsP(g))=\emptyset\\
      \inf\limits_{\psi\in H(\bsP(f),\bsP(g))}\!\!  \{\norm{\psi-id }_{\cC^0(X,X)}+\norm{f-g\circ\psi}_{\cC^0(X,X)} \}   &  \text{ otherwise,}
     \end{cases}
\]
where \(A=2 \operatorname{diam}(X)\) and \(id\) is the identity function, that is \(id(x)=x\). \\
     Evidently \(\rho(f,g)\leq A\) for any piecewise contractions \(f,g\). Further notice that the metric (proved in Lemma~\ref{lem:rho}) \(\rho\) is essentially a distance between two tuples \((f,\bsP(f))\) and \((g,\bsP(g))\) for any two piecewise contractions \(f,g\) where \(\bsP(f)\) and \(\bsP(g)\) are the maximal partitions of \(f\) and \(g\) respectively.\\
	  For an arbitrary \(\sigma\in (0,1)\) we define the distance \(d_1\) as 
         \[
		 d_1(f,g)=\sum_{n=1}^{\infty}\sigma^{n}\rho(f^n,g^n).
		 \]

     Under this metric, we have the following result (proved in Section~\ref{sec:op}):
		 \begin{thm} \label{thm:op}
			The set of Markov piecewise contractions is open in the set of piecewise contractions under the metric \(d_1\). Moreover, any two Markov piecewise contractions \(f,\ g\) close enough (w.r.t. \(d_1\)), stabilise at the same time.
		\end{thm}
 To be able to state our stability result, we need to discuss the dynamics of the partition elements.
	 \begin{defin}[Wandering Set]
		For a piecewise contraction \(f:X\to X\), a partition element \(P\in\bsP(f)\) is called Wandering if there exists \(M\in\bN\) such that \(f^n(P)\cap P=\emptyset\) for all \(n>M\). The set \(\bW(f)\subset \bsP(f)\), consisting of all wandering partition elements, is called the \emph{wandering} elements set. We set $W(f)=\cup_{P\in\bW(f)}P$.\footnote{ Note that this set is not the usual wandering set, which is much bigger.}
   \end{defin}
		
   \begin{defin}[Non-wandering Set]
  The complement \(N\bW(f)\) of the Wandering elements set is called the non-wandering elements set, that is, \(N\bW(f)= \{P\in \bsP(f)\::\; P\not \in\bW(f)\}\). We set $NW(f)=\cup_{P\in N\bW(f)}P$.
   \end{defin}
 Note that, for a Markov map, the set $N\bW(f)$ correpsonds to the non-wandering set of the dynamical system $(\Omega(f),\bbf)$ defined in Remark \ref{rem:symbolic_dyn}. Accordingly, our definition of $NW(f)$ should not be confused with the usual non-wandering set of $f$, which, for a Markov map, consists of finitely many points (the set of periodic points, see Theorem \ref{thm:eq}). Hereby we state our definition of topological stability:
\begin{defin}[Topological Stability]{\label{def:ss}}
	 Let \(\mathfrak{C}\) be contained in the set of piecewise contractions from $X\to X$, and \(d \) be a metric defined on the piecewise contractions. We say that \((\mathfrak{C},d)\) is \emph{topologically stable}  if for every \(f\in\mathfrak{C}\), there exists a \(\delta>0\) such that, for any piecewise contraction \(g\) with \(d(f,g)<\delta\), \(g\in\mathfrak{C}\), moreover \(f\) is semi-conjugate to \(g\) and \(g\) is semi-conjugate to \(f\) on the non-wandering sets, that is, there exist continuous functions \(H_1: {NW}(f)\to{NW}(g),\  H_2:{NW}(g)\to{NW}(f)\)  such that \(H_1\circ f= g\circ H_1\), \(f\circ H_2= H_2 \circ g\), and $H_1(NW(f))=NW(g)$, $H_2(NW(g))=NW(f)$.
\end{defin}
	 This definition of topological stability is somewhat different from the standard definition (see \cite[Defintion 2.3.5]{KH95}). In general, topological stability is defined for homeomorphisms  such that one of them is semi-conjugate to the other (whereas we have semi-conjugacies for both sides), but on the entire space. Here it is necessary to define it only on the non-wandering set which, in fact, is the only subset of the space that contributes to the long-time dynamics. We can interpret our definition as the one stating that the long-time dynamics of two such functions are qualitatively the same. 
		\begin{thm} \label{thm:ss} 
			The set of Markov piecewise injective contractions (as defined in Definitions \ref{def:markov}, \ref{def:pic}) is topologically stable in the \(d_1\) topology.
		\end{thm}
		The proof of the above theorem is given in Section~\ref{sec:op}. To state our result on density, we restrict ourselves to piecewise smooth contractions, defined as follows:
\begin{defin}[Piecewise smooth contraction]\label{def:psc}
	 A piecewise injective contraction \(f\) with partition \(\bsP(f)=\{P_i\}, i\in\{1,2,\ldots,m\}\) and injective extensions \(\tilde f_i:U_i\to\bR^d\), is called a \emph{piecewise smooth contraction} if, for every \(i\in\{1,2,\ldots,m\}\),
\begin{itemize}
	 \item  \(\norm{\tilde f_i}_{\cC^3}<\infty\), 
	 \item \(\norm{\tilde f_i^{-1}}_{\cC^1}<\infty\).
	 \item \(\partial P_i\) is contained in the union of finitely many \(\cC^2\) \((d-1)\)-dimensional manifolds $\{M_j\}$, $\partial M_j\cap \overline P_i=\emptyset$. Such manifolds are pairwise transversal, and the intersection of any set of such manifolds consists of a finite collection of  \(\cC^2\) manifolds.
\end{itemize} 
\end{defin}
      For a piecewise smooth contraction \(f\), we define the extended-metric
	 \[
		d_2(f,g)=\begin{cases} \sup\limits_{P\in\bsP(f)}\norm{f-g}_{\cC^2(P)} &\text{if}\ \bsP(f)=\bsP(g)\\
			\ \ \infty & \text{otherwise}.
		\end{cases}
	 \]
 Note that $d_2(f,g)\geq \rho(f,g)$. The following density result  is proven in Section~\ref{sec:den}:

\begin{thm}\label{thm:den}
 	 Piecewise smooth Markov contractions are $d_2$-dense in the space of piecewise smooth contractions.
	 \end{thm}
  
\begin{rem}	 
	 Theorem~\ref{thm:eq}, Theorem~\ref{thm:op}, Theorem~\ref{thm:ss} show that, for a piecewise contraction, to be Markov is stable under a rather weak topology (\(d_1\)). Instead, Theorem~\ref{thm:den} shows that Markov is dense under a quite strong topology (\(d_2\)). These theorems collectively show that, for a piecewise contraction, being Markov is generic; hence, to have a Cantor set as an attractor is rare. 
\end{rem}
	 As already mentioned, a result on density is present in literature, see \cite{CGMU}. However, it is proved under a much coarser metric as compared to \(d_2\). More importantly, it assumes that the maps have the separation property, which implies they are globally injective, whereas we assume only piecewise injectivity.

\section{Openness and Topological stability} \label{sec:op}
     Recall that, for any piecewise contraction \(f\),  \(\bsP(f)\) stands for the maximal partition, so \(\partial\bsP(f)=\Delta(f)\). In addition, for any \(n>1\), the elements of partition \(\bsP(f^n)\) are given by the equation~(\ref{eq:n-partn}), and \(\#\bsP(f^n)=m_n\), \(m_1=m\).
\begin{proof}[\bf\em{Proof of Theorem~\ref{thm:eq}}] 
     Let \(f\) satisfy \(\Lambda(f)\cap\Delta(f)=\emptyset\). For all \(n\in\bN\), \(\overline{f^{n+1}(X)}\subset \overline{f^{n}(f(X))}\subset \overline{f^n(X)}\) implies that \(\{\overline{f^n(X)}\}\) is a nested sequence of non-empty compact sets. Further, Cantor's intersection theorem implies that \(\Lambda(f)=\cap_{n\in \bN} \overline{f^n(X)}\neq \emptyset\) and it is closed.  Accordingly, there exists \(\ve>0\), such that \(d_0(\Lambda(f),\Delta(f))>\ve\). We claim that, for any $\ve>0$,  there exists \(N_{\ve}\in \bN\)  such that for \(n\geq N_{\ve}\), \(\overline{f^n(X)}\subset B_{\ve/2}(\Lambda(f))\).\footnote{For \(r>0\) and a set \(A\), \( B_r(A)=\{y\;:\; \exists\  x\in A \textrm{ such that }d_0(x,y)<r \)\}.} Indeed, if this was not the case, then there would exist a sequence $\{n_j\}$, $n_j\to\infty$, such that \(\overline{f^{n_j}(X)}\cap B_{\ve/2}(\Lambda(f))^c\neq \emptyset\). It follows that for each $n\in \bN$ there exists $j$ such that $n_j\geq n$, hence 
\[
\overline{f^{n}(X)}\cap B_{\ve/2}(\Lambda(f))^c\supset  \overline{f^{n_j}(X)}\cap B_{\ve/2}(\Lambda(f))^c\neq \emptyset,
\] 
which, taking the intersecion on $n$, yileds a contradiction.
Consequently, for every \(P\in\bsP(f^{N_\ve})\), there exists \(Q \in \bsP(f^{N_{\ve}})\) such that \(\overline{f^{N_\ve}(P)}\subset Q\), otherwise there would exist \(x\in f^{N_\ve}(P)\cap \partial\bsP(f^{N_{\ve}})\), that is, a \(k\in\bN\) such that 
	     \[
		 f^{k}(x)\in f^{k+N_{\ve}}(P)\cap \partial \bsP(f)\subset \overline{f^{k+N_\ve}(X)}\cap\Delta(f)\subset B_{\ve/2}(\Lambda(f))\cap\Delta(f)=\emptyset,
		 \]
	 which is a contradiction.\\
     Conversely, let \(f\) be a piecewise Markov contraction with stabilisation time \(N\in\bN\). By definition, for every \(P\in\bsP(f^N)\), there exists \(Q\in\bsP(f^N)\) such that \(\overline{f^N(P)}\subset Q\), that is \(\overline{f^N(P)}\cap\partial \bsP(f^N)=\emptyset\). Note that  \(\overline{f(X)}=\bigcup_{i=1}^m\overline{f(\tilde P_i)}\subset\bigcup_{i=1}^m \overline{f(P_i)} \), where $\tilde P_i$ is as given in Definition~\ref{def:pc}. Thus,
	     \[
	     \Lambda(f)=\bigcap_{n\in\bN}\overline{f^n(X)}\subset \bigcap_{n\in\bN}\bigcup_{P\in\bsP(f^n)}\overline{f^n(P)}\subset \bigcup_{P\in\bsP(f^N)}\overline{f^N(P)}\subset \bigcup_{Q\in \bsP(f^N)}Q
		 \]
	 implies \(\Lambda(f)\cap\partial\bsP(f)=\emptyset\).\\ 
     To prove the second part of the Theorem, let \(x\in \Lambda(f)\) and $N$ be the stabilisation time. By definition, for each $k\in\bN$, there exists \(y_k\in Q_{k}\in \bsP(f^N)\) such that \(x=f^{kN}(y_k)\). In addition, there exists \(P\in\bsP(f^N)\) such that \(x\in P\). This implies \({\overline{f^{kN}(Q_k)}}\subset P\) for all \(k\in\bN\). Let \(l:=\# \bsP(f^N)\),\footnote{ For a discrete set \(M,\  \#  M\) denotes the cardinality of \(M\).} then there exists $k_1\in\{0,\dots, l\}$ such that \(P=Q_{k_1}\) hence $\overline{f^{k_1N}(P)}\subset P$.
	 By Contraction Mapping Theorem \(f^{k_1N}:\overline{P}\to \overline{P}\) has a unique fixed point say \(z\in \overline{P}\). Let $j\in\bN$ be the smallest integer for which $f^j(z)=z$, then $x\in  \bigcap_{n\in\bN}\overline{ f^{nj}(P)}=\{z\}$.  Hence, \(\Lambda(f)\) consists of periodic orbits. 
\end{proof}
To prove the result on openness and topological stability, we first prove that the functions \(\rho,\ d_1\), defined in Section~\ref{sec:results}, are in fact metrics on the set of piecewise contractions. Note that, by definition of \(H(\bsP(f),\bsP(g))\), if \(\# \bsP(f)\neq \# \bsP(g)\) then \(H(\bsP(f),\bsP(g))=\emptyset\).
\begin{lem} {\label{lem:rho}}
	 \(\rho\) is a metric. 
\end{lem}
\begin{proof}
Let \(f,g,h\) be piecewise contractions.\\
If \(\rho(f,g)=0\) then there exists a sequence \(\psi_n\in H(\bsP(f),\bsP(g))\), \(\norm{\psi_n-id}_{\cC^0(X,X)} \to 0\), and \(\norm{f-g\circ\psi_n}_{\cC^0(X,X)}\to 0\) as \(n\to \infty\) which implies \(\psi_n\to id\)  as \(n\to\infty\) which further implies that \(\bsP(f)=\bsP(g)\). Also, \(\psi_n\to id\) means that \(f-g\circ \psi_n\to f-g\), hence \(f=g\).\\
Next, we check the symmetry of \(\rho\):
 if \(H(\bsP(f),\bsP(g))=\emptyset\) then \(\rho(f,g)=\rho(g,f)=A\). If \(H(\bsP(f),\bsP(g))\neq\emptyset\) then
\begin{align*}
     \rho(f,g)&=\inf\limits_{\psi\in H(\bsP(f),\bsP(g))} \{\norm{\psi-id }_{\cC^0(X,X)}+\norm{f-g\circ\psi}_{\cC^0(X,X)}\}\\
     & = \inf\limits_{\psi^{-1}\in H(\bsP(g),\bsP(f))} \{\norm{\psi^{-1}-id}_{\cC^0(X,X)}+\norm{f\circ\psi^{-1}-g}_{\cC^0(X,X)}\}\\
     &=\rho(g,f).
\end{align*}
It remains to check the triangle inequality:\\
     To show that $\rho(f,g)\leq \rho(f,h)+\rho(g,h)$, consider the following cases:\\
	 if \(H(\bsP(f),\bsP(g))=\emptyset\) then \(\rho(f,g)=A\), and
	  \(H(\bsP(f),\bsP(h))=\emptyset\) or/and \(H(\bsP(g),\bsP(h))=\emptyset\), that is, either one of the two or both are empty sets. Therefore, \(\rho(f,h)=A\) or \(\rho(g,h)=A\) and so
		 \(\rho(f,g)=A \leq \rho(f,h)+\rho(h,g)\).\\
 If \(H(\bsP(f),\bsP(g))\neq\emptyset\) then there are the following two possibilities:
\begin{enumerate}
	 \item 
	 \(H(\bsP(f),\bsP(h))=\emptyset\) and \(H(\bsP(h), \bsP(g))=\emptyset\) then\\
	 \(\rho(f,h)+\rho(g,h)\geq A\) and 
\begin{align*} 
	 \hspace{1.1cm}\rho(f,g)& =\inf\limits_{\psi\in H(P(f),P(g))} \{\norm{\psi-id}_{\cC^0(X,X)}+\norm{f-g\circ\psi}_{\cC^0(X,X)}\}\\
	 &\leq2\ \operatorname{diam}(X)
\end{align*} 
	 Since \(A\geq2\ \operatorname{diam}(X)\), we have the result.
	 \item  \(H(\bsP(f),\bsP(g))\neq\emptyset\), and \(H(\bsP(g),\bsP(h))\neq\emptyset\).\\
	  Given \(\phi\in H(\bsP(g),\bsP(h))\) and \(\varphi\in H(\bsP(f),\bsP(h))\), the homeomorphism \(\psi=\phi^{-1}\circ\varphi \in H(\bsP(f), \bsP(g))\), and hence
	 \begin{align*}
	 \rho(f,g)&=\inf\limits_{\psi\in H(\bsP(f), \bsP(g))}\{\norm{\psi-id}_{\cC^0(X,X)}+ \norm{f-g\circ\psi}_{\cC^0(X,X)}\}\\
	 &\leq \inf\limits_{\varphi\in H(\bsP(f),\bsP(h))}\inf\limits_{\phi\in H(\bsP(g),\bsP(h))} \{\norm{\phi^{-1}\circ\varphi-id}_{\cC^0(X,X)}+\norm{f-h\circ\varphi}_{\cC^0(X,X)}\\
	 &\ \ \ +\norm{h\circ\varphi-g\circ\phi^{-1}\circ\varphi}_{\cC^0(X,X)}\}  \\
	 &\leq \inf\limits_{\varphi\in H(\bsP(f),\bsP(h))}\inf\limits_{\phi\in H(\bsP(g),\bsP(h))} \{\norm{\phi^{-1}-id}_{\cC^0(X,X)}+\norm{f-h\circ\varphi}_{\cC^0(X,X)}\\
	 &\ \ \ +\norm{\vf-id}_{\cC^0(X,X)}+\norm{h\circ\varphi-g\circ\phi^{-1}\circ\varphi}_{\cC^0(X,X)}\}  \\
	 &=\inf\limits_{\phi\in H(\bsP(g),\bsP(h))} \{\norm{\phi^{-1}-id}_{\cC^0(X,X)}+\norm{h\circ\varphi-g\circ\phi^{-1}\circ\varphi}_{\cC^0(X,X)}\}\\
	 &\ \ \ + \inf\limits_{\varphi\in H(\bsP(f),\bsP(h))}\{\norm{\vf-id}_{\cC^0(X,X)}+\norm{f-h\circ\varphi}_{\cC^0(X,X)}\}\\
	 &=\rho(f,h)+\rho(g,h)
\end{align*}
\end{enumerate}
     Hence \(\rho(\cdot,\cdot)\) is a metric.	
\end{proof}
     Lemma~\ref{lem:rho} implies that \(d_1\) is also a metric since the series is convergent.

\begin{proof} [\bf\em Proof of Theorem~\ref{thm:op}]
	 Let \(f \) be Markov, we want to prove that there exists a neighbourhood of \(f\) consisting of only Markov contractions. Since \(f\) is Markov, there exists \(N\in\bN\) such that the maximal partition \(\bsP(f)\) of \(f\) stabilises with stabilisation time \(N\). Let \(g\neq f\) be a piecewise contraction such that \(d_1(f,g)<\delta\) for\footnote{Recall that \(A=\operatorname{diam}(X)\).}  \(0<\delta< \sigma^N A\).\\
 We will show that the partition of \(g\) stabilises with the same stabilisation time \(N\).\\
Note that, for all \(n\leq N\), \(\rho(f^n,g^n)<A\) which implies \(H(\bsP(f^n),\bsP(g^n))\neq\emptyset\). 
	 Let \(P\in\bsP(g^N)\) then there exist \(\psi_N\in H(\bsP(f^N),\bsP(g^N))\) and \(P^\prime\in\bsP(f^N)\) such that \(\psi_N(P^\prime)=P\). By stabilisation, for \(P^{\prime}\), there exists a unique \(Q^\prime\in\bsP(f^N)\) such that \(\overline{f^N(P^\prime)}\subset Q^\prime\), also, \(\psi_N(Q^\prime)=Q\) for some \(Q\in\bsP(g^N)\). Now, \(\overline{f^N(P^\prime)}\subset Q^\prime\) and \(Q^\prime\) being open implies  \(\ve=\min_{P^\prime\in\bsP(f^N)}d_0(\overline{f^N(P^\prime)}, \partial Q^\prime)>0\).
Choosing \(0<\delta<\ve\sigma^N/3\), we claim that \(\overline{g^N(P)}\subset Q\). Indeed, if there exists \(x\in \overline{g^N(P)}\cap \partial Q\), then there exists a sequence \(\{x_k\}\in g^N(P)\cap Q\) such that \(x_k\to x\) as \(k \to \infty\). Let \(y_k\in P\) such that \(g^N(y_k)=x_k\), note that \(\psi_N^{-1}(x)\in \partial Q^\prime\). Now 
\begin{align*}
	 \ve&<d_0(f^N(\psi_N^{-1}(y_k)),\psi_N^{-1}(x))\\
	 &\leq d_0(f^N(\psi_N^{-1}(y_k)),g^N(y_k))+d_0(g^N(y_k),x)+d_0(x,\psi_N^{-1}(x))\\
	 &\leq \norm{f^N\circ \psi_N^{-1}-g^N}_{\cC^0(X,X)}+\norm{id-\psi_N^{-1}}_{\cC^0(X,X)}+ d_0(x_k,x)\\
	 &= \rho(f^N,g^N)+d_0(x_k,x)\quad(\text{taking infimum over}\ \psi_N\in H(\bsP(f^N),\bsP(g^N))\ \text{on both sides})\\
	 &<\delta \sigma^{-N}+ d_0(x_k,x) \to \ve\sigma^N\sigma^{-N}/3 \ \text{as} \ k\to \infty,
\end{align*} 
     we get \(\ve<\ve/3\) which is a contradiction. Hence, if \(f\) is Markov with \(N\) as the stabilisation time, then for
	  \[
		\delta<\min\left\{\sigma^N A, \frac{\sigma^N}{3}\inf_{P,Q\in\bsP(f^N)}d_0(\overline{f^N(P)},\partial Q)\right\},
	 \]
all piecewise contractions \(g\), with \(d_1(f,g)<\delta\), are Markov contractions and the stabilisation time of \(g\) is also \(N\). Thus, the collection of Markov contractions is open.
\end{proof}
 
To prove Theorem~\ref{thm:ss}, we first need to prove the following lemma which in itself brings some important information about the dynamics of a Markov contraction.

\begin{lem}\label{lem:refined_partition}
	Let \(f\) be a Markov contraction with maximal partition \(\bsP(f)\) and stabilisation time \(N\), then for every \(P\in\bsP(f^N)\), there exists \(Q\in\bsP(f^N)\) such that \(f(P)\subset Q\).
\end{lem}
\begin{proof}
	By Defintion~\ref{def:stp}, there exists \(P^\prime\in \bsP(f^N)\) such that \(f^N(P)\subset P^\prime\). By the defintion given in Equation~(\ref{eq:n-partn}), there exist parition elements \(\{P_i\}_{i=0}^{N-1}\) (non necessarily distinct) in \(\bsP(f)\) such that
	\[
	 P=	P_0\cap f^{-1}P_1\cap f^{-2}P_2\cap\ldots\cap f^{-(N-1)}P_{N-1}.
	\]
	Similarly there exist partition elements \(\{P_j^\prime\}_{j=1}^{N-1}\) in \(\bsP(f)\) such that 
    \[
	 P^\prime = P_0^\prime \cap f^{-1}P_1^\prime \cap f^{-2}P_2^\prime\cap\ldots\cap f^{-(N-1)}P_{N-1}^\prime.
    \]
	Let \(x\in P\) then \(f^k(f(x))\in P_{k+1}\), for every \(k\in \{0,1,\ldots,N-2\}\) and \(f^N(x)= f^{N-1}f(x)\in P^\prime\) which implies \(f(x)\in f^{-(N-1)}(P^\prime_0)\). Consequently,
	\[
		f(x)\in P_1\cap f^{-1}P_2\cap\ldots\cap f^{-(N-2)}P_{N-1}\cap f^{-(N-1)}P^\prime_0=Q\in \bsP(f^N)
	\] 
	Since \(x\in P\) is arbitrary, \(f(P)\subset Q\in\bsP(f^N)\).
\end{proof}
 
 Finally, to prove Theorem~\ref{thm:ss}, we restrict to piecewise injective contractions and prove the stability result under the metric \(d_1\). Recall that for every piecewise injective contraction \(f\), with maximal partition \(\bsP(f)=\{P_1,\ldots,P_m\}\), there exists \(U_i=\mathring U_i\supset \overline{P_i}\) and an injective continuous extension \(\tilde{f}:U_i\to\bR^d\) such that \(\tilde f|_{P_i}=f|_{P_i}\). In this proof, we always consider these extensions which, to alleviate notation, we still denote by $f$.\\
 Our strategy for the following proof is as follows. For piecewise injective Markov contractions \(f,g\) with maximal partitions \(\bsP(f), \bsP(g)\) respectively and, for some \(\delta>0\), \(d_1(f,g)<\delta\), we construct semi-conjugacies from \(NW(f)\), the non-wandering set for \(f\), to \(NW(g)\), the non-wandering set for \(g\), using a homeomorphism between the partitions given in the definition of the metric \(\rho\).  
\begin{proof}[\bf\em{Proof of Theorem~\ref{thm:ss}}] 
\label{sec:2.14}
     
     Let \(N\in\bN\) be the stabilisation time of \(\bsP(f)\). By Theorem~\ref{thm:op}, there exists \(\delta>0\) such that if  \(d_1(f,g)<\delta\), then \(\bsP(g)\) also stabilises at time \(N\). By Theorem~\ref{thm:eq}, the attractors of \(f\) and \(g\) consist of eventually periodic orbits. Let \(P\in \bsP(f^N)\) be a periodic element of the partition, then there exists \(n_0\in \bN\) such that \(\overline{f^{n_0}(P)}\subset P\). Then \(d_1(f,g)<\delta\) implies that there exists \(Q\in\bsP(g^N)\) such that \(\overline{g^{n_0}(Q)}\subset Q\). By the Contraction Mapping Theorem, for \(f^{n_0}: \overline{P}\to \overline{P}, g^{n_0}: \overline{Q}\to \overline{Q}\), there exist \(x_f\) and \(x_g\), the unique fixed points of \(f^{n_0}\) and \(g^{n_0}\) respectively in \(P\) and \(Q\).\\
	 Using Lemma~\ref{lem:refined_partition} inductively, let \(P_i,\  Q_i\) be the partition elements in \(\bsP(f^N), \bsP(g^N)\) respectively, such that \( f^i(\overline P)\subset P_i \), \( g^i(\overline Q)\subset Q_i\). Let \(\widehat{P}= \cup_{i=1}^{n_0} {\overline P_i}\), \( \widehat{Q}=\cup_{i=1}^{n_0}{\overline Q_i}\), then for each \(i\in \bN\), \(\widehat{P}_i=f^i(\widehat{P})\subset \widehat{P}\), \(\widehat{Q}_i=g^i(\widehat{Q})\subset \widehat{Q} \).\\
     Note that for \(\delta\) small enough, we have \( H(\bsP(f^n),\bsP(g^n))\neq \emptyset\) for every \(n\leq n_0\), that is, there exists \(\psi \in H(\bsP(f^{n_0}),\bsP(g^{n_0}))\) such that \(\norm{\psi-id}_{\cC^0(X,X)}<\delta\) and  \(\norm{f^n-g^n\circ\psi}_{\cC^0(X,X)}<\delta\), and thus  \(\psi(\widehat{P})=\widehat{Q}\).\\
      Next, define $\widehat g=\psi\circ f\circ \psi^{-1}$. Since, by Definition \ref{def:pic}, $f^{-1}$ is well defined on $\widehat P$, we have that $\widehat g$ is invertible on $\widehat Q$ and $\widehat g^{-1}(\partial \widehat Q)=\psi\circ f^{-1}( \partial \widehat P)$.\\
     Next, for $\ve>0$ small enough, let $\widehat P_{1, \ve}$ be the $\ve$-neighborhood of $f(\widehat P)$ and $\widehat P_{c,\ve}$ be the $\ve$-neighborhood of ${\widehat P}^\complement$, the complement of \(\widehat P\). Similarly, let $\widehat Q_{c,\ve}=\psi\circ f^{-1}(\widehat P_{c,\ve})$ and $\widehat Q_{1,\ve}=\psi\circ f^{-1}(\widehat P_{1,\ve})$. By the Markov property, we have $\overline{\widehat P}_{1,\ve}\cap \overline{\widehat P}_{c,\ve}=\emptyset$, and thus \(\overline{\widehat Q}_{1,\ve}\cap \overline{\widehat Q}_{c,\ve}=\emptyset\). 
	 Hence, by Urysohn's lemma there exists a function $\theta\in\cC^0(X,[0,1])$ such that $\theta|_{\overline{\widehat Q}_{1, \ve}}=1$ and $\theta|_{\overline{\widehat Q}_{c,\ve}}=0$. Finally, define the following continuous functions:
     \[
     \begin{split}
     &\tilde g(x)=\theta(x) g(x)+(1-\theta(x))\widehat g(x)\\
     &h_0(x)=\tilde g\circ \psi\circ f^{-1}(x), \quad \forall x\in \overline{\widehat P\setminus\widehat P_1}.
     \end{split}
     \]
     \begin{lem} \label{lem:surj}
     Provided $\delta>0$   small enough, we have $h_0(\widehat P\setminus\widehat P_1)=\widehat Q\setminus\widehat Q_1$, $h_0(\partial \widehat P)=\partial \widehat Q$, $h_0(\partial \widehat P_1)=\partial \widehat Q_1$.
     \end{lem}
     \begin{proof}
      Firstly, using the properties of the homeomorphism \(\psi\), we have for each $x\in \overline{\widehat P\setminus\widehat P_1}$, 
     \begin{equation}\label{eq:identity}
     |h_0(x)-x|=\Big| \theta(x) [g\circ \psi\circ f^{-1}(x)-x]+(1-\theta(x))[\psi(x)-x]\Big|\leq 2\delta.
     \end{equation}
	 If $x\in \widehat P_{c,\ve}\cap\widehat P$, then \(\psi\circ f^{-1}(x)\in \widehat Q_{c,\ve}\), thus $h_0(x)=\widehat g\circ \psi\circ f^{-1}(x)=\psi(x)$. While, if $x\in \widehat P_{1,\ve}$, then \(\psi\circ f^{-1}(x)\in \widehat Q_{1,\ve}\), and thus $h_0(x)=g\circ \psi\circ f^{-1}(x)$.
     In addition,
     \[
     \begin{split}
     &h_0(\partial \widehat P)=\psi(\partial \widehat P)=\partial \widehat Q\\
     &h_0(\partial \widehat P_1)=g\circ \psi\circ f^{-1}(\partial\widehat P_1)=g\circ \psi(\partial\widehat P)=g(\partial \widehat Q)=\partial \widehat Q_1.
     \end{split}
     \]
     Also, if $\delta$ is small enough, then $h_0$ is invertible on $(\widehat P_{c,\ve}\cup\widehat P_{1, \ve})\cap \widehat P$.
     Thus, to prove surjectivity it suffices to prove that each $p\in (\widehat Q\setminus\widehat Q_1)\setminus h_0((\widehat P_{c,\ve}\cup\widehat P_{1, \ve})\cap \widehat P)$ belongs to $h_0(\widehat P\setminus \widehat P_1)$. Let $B=\{z\in\bR^{n}\;:\; \|z\|\leq 3\delta\}$, $x=p+z$ and $h_0(x-p)=x-h_0(x)$. Then $h_0(x)=p$ is equivalent to
     \[
     z=h_0(z).
     \]
     Since equation \eqref{eq:identity} implies that $h_0(B)\subset B$, by Brouwer fixed-point Theorem it follows that there exists at least one $z\in B$ such that $h_0(z+p)=p$ and, for $\delta\leq \ve/6$, $z+p\in \widehat P\setminus \widehat P_1$.
     \end{proof}

    For the sake of convenience, let \(\widehat P_0=\widehat P\) and \(\widehat Q_0=\widehat Q\). For every \(i\in \bN\), define \(h_{i}:\overline{\widehat{P}_i\backslash \widehat{P}_{i+1}}\to \overline{\widehat{Q}_i\backslash \widehat{Q}_{i+1}}\) as \(h_{i}(x)=g\circ h_{i-1}\circ f^{-1}(x)\).
     Thus, define the semi-conjugacy \(H:\overline{\widehat{P}}\to \overline{\widehat{Q}}\) as 
	     \[
		 H(x)=\begin{cases}
         h_{i}(x), \quad\quad x\in \overline{\widehat{P}_i\backslash \widehat{P}_{i+1}}\\
         x_g, \quad\quad \quad\quad\ \  x=x_f.
         \end{cases}
		 \]
     \(H\) is continuous because for \(x\in\partial \widehat{P}_i, H(x)=h_{i+1}(x)=g\circ h_i\circ f^{-1}(x)=h_i(x)\) and for any sequence \((x_i)\in \widehat{P}_i\backslash \widehat{P}_{i+1}\) with \(x_i\to x_f\) as \(i\to\infty\), \((Hx_i)\in \widehat{Q}_i\backslash \widehat{Q}_{i+1}\), so \(H(x_i) \to x_g=H(x_f)\). Indeed \(H\) is surjective, for \(p\in \overline{\widehat{Q}}\), there exists \(i\in\bN\cup\{0\}\) such that \(p\in \overline{\widehat Q_i\setminus\widehat Q_{i+1}}\), we use induction on \(i\). If \(i=0\) then using Lemma~\ref{lem:surj}, there exists \(z\in \overline{\widehat Q_i\setminus\widehat Q_{i+1}}\) such that \(h_0(z)=p\). Instead if \(i\neq 0\) then \(p_1=g^{-1}(p)\in \overline{\widehat Q_{i-1}\setminus\widehat Q_{i}}\), then inductively \(h_{i-1}\) being surjective, there exists \(q_1\in \overline{\widehat P_{i-1}\setminus\widehat P_{i}}\) such that \(h_{i-1}(q_1)=p_1\). Now, by definition, \(q=f(q_1)\in \overline{\widehat P_{i}\setminus\widehat P_{i+1}}\), and finally \(g\circ h_i\circ f^{-1}(q)=p\). \\
	 Furthermore, \(H\) is the wanted semi-conjugacy between \(f,\ g\), on \(\widehat{P}\) and \(\widehat{Q}\), because for \(x\in \widehat{P}\), 
	     \[
		 H_{\widehat{P}}\circ f(x)= h_{i+1}\circ f(x)=g\circ h_i\circ f^{-1}\circ f(x)=g\circ h_i(x) =g\circ H(x).
		 \]
	 To obtain a semi-conjugacy on the whole of \(NW(f)\), we repeat the same steps for every periodic partition element \(P\in \bsP(f^N), Q\in \bsP(g^N)\) with the same \(\psi\in H(\bsP(f^N),\bsP(g^N))\). Calling $\{\widehat{P}^k\}$ the collection of the union of elements associated to a periodic orbit, we have $\cup_k \widehat{P}^k=NW(f)$. Pasting these functions together, we obtain a function \(\Theta:{NW}(f)\to {NW}(g)\) with  $\Theta(x)=H_{\widehat{P}^k}(x)$ for $ x\in \widehat{P}^k\subset NW(f)$. Then \(\Theta\) is continuous on \(NW(f)\) as \(\Theta|_{\partial\widehat{P}^k}=\psi|_{\partial\widehat{P}^k}\) for every \(k\).\\
     To construct the semi-conjugacy from the other side, we use the fact that \(\psi\) is a homeomorphism and repeat the same construction using $\psi^{-1}$ instead of $\psi$ and switching the roles of $f$ and $g$.
		\end{proof}
 \section{Density} \label{sec:den}
	 We will prove Theorem~\ref{thm:den} in two steps. First we show that Markov maps are dense in a special class of systems (piecewise strongly contracting) and then we will show that such a class is itself dense in the collection of piecewise smooth contractions.
	 
 \begin{defin}[Piecewise strongly contracting]
			 A piecewise smooth contraction \(f\) with contraction coefficient \(\lambda\) and maximal partition \(\bsP(f)=\{P_1,P_2,\ldots,P_m\}\) is said to be piecewise strongly contracting if there exists \(p\in\bN\) such that \(\lambda^pm_p<1/2\), where \(m_p=\#\bsP(f^p)\).
 \end{defin} 
	 We will prove the following results:
 \begin{prop} \label{prop:markov-dense}
	 Markov maps are \(d_2\)-dense in the collection of  piecewise strong contractions.
 \end{prop} 
 \begin{prop} \label{prop:psc-dense}
	 Piecewise strong contractions are $d_2$-dense in the collection of piecewise smooth contractions.
 \end{prop}
	 These two Propositions readily imply our main result.
	 \begin{proof}[\bf \em Proof of Theorem~\ref{thm:den}]
	 Let \(f\) be a piecewise smooth contraction. By Proposition~\ref{prop:psc-dense}, for each \(\ve>0\), there exists a piecewise strong contraction \(f_1\) such that \( d_2(f,f_1)<\ve/2\). In addition, by Proposition~\ref{prop:markov-dense}, there exists a  piecewise smooth Markov contraction \(f_2\) such that \( d_2(f_1,f_2)<\ve/2\), hence the result.
 \end{proof}
	 
In the rest of the paper, we prove Proposition~\ref{prop:markov-dense} and Proposition~\ref{prop:psc-dense}. 

The basic idea of the proof is to introduce iterated function systems (IFS) associated with the map. The attractor of the IFS is greater than the one of the map (see section \ref{sec:ITF}  for the relation between the two sets) hence if we can prove that the attraction of the IFS is disjoint from the discontinuities of the map, so will be the attractor fo the map. The advantage is that, in this way, the study of the boundaries of the elements of $\bsP(f^n)$ is reduced to the study of the pre-images of the discontinuities of $f$ under the IFS. Hence we can iterate smooth maps rather than discontinuous ones. 

This advantage is first exploited in the section \ref{sec:4.2}, where we prove Proposition~\ref{prop:markov-dense} using an argument that is, essentially, a quantitative version of Sard's theorem.

To prove Proposition~\ref{prop:psc-dense}, the rough idea is to use a transversality Theorem (see Appendix \ref{sec:transversality}) to show that if a lot of pre-images intersect, then, generically, their intersection should have smaller and smaller dimensions till no further intersection is generically possible. Unfortunately, if we apply a transversality theorem to a composition of maps of the IFS, we get a perturbation of the composition and not of the single maps. How to perturb the single maps in such a way that the composition has the wanted properties is not obvious.

Our solution to this problem is to make sure that if we perturb the maps in a small neighbourhood $B$, and we consider arbitrary compositions of the perturbed maps, then all the images of $B$ along the composition never intersect $B$. Hence, if we restrict the composition to $B$, all the maps, except the first, will behave as their unperturbed version. To ensure this, it suffices to prove that such compositions have no fixed points near the singularity manifolds (such implication is proven in Lemma \ref{lem:fix}).
To this end, in Propositions \ref{prop:fix-sep} and \ref{prop:bad_points}, we show that one can control the location of the fixed points of the compositions of the map of the IFS by an arbitrarily small perturbation.

After this, we can finally set up an inductive scheme to ensure that the pre-images of the discontinuity manifolds keep intersecting transversally. This is the content of Proposition \ref{lem:good_partn} from which Proposition~\ref{prop:psc-dense} readily follows.

\subsection{IFS's associated to the map and their properties}\ \\
\label{sec:ITF}
We start by recalling the definition of IFS relevant to our argument and exploring some of its properties.
 \begin{defin}[Iterated Function System]
	 The set \(\Phi=\{\phi_1, \phi_2. \ldots,\phi_m\}, m\geq2\), is an Iterated Function System (IFS) if each map \(\phi_i:\bR^d\to \bR^d\) is a Lipschitz contraction.\footnote{ In the following we will consider only $\cC^3$ maps $\phi_i$.}
 \end{defin}

Let \(f\) be a piecewise smooth contraction with maximal partition \(\bsP(f)\). In analogy with \cite{NPR2},  we define an Iterated Function System associated to \(f\) as follows:\\ 
	 By Definition \ref{def:psc}, \(\tilde f|_{U_i}\) is \(\cC^3\)  for every \(i\) and \(D\tilde f|_{U_i}\leq\lambda<1\), using the \(\cC^r\) version of Kirszbraun-Valentine Theorem~\ref{thm:sKV}, we obtain a \(\cC^3\) extension \(\phi_i:\bR^d\to\bR^d\) of \(\tilde f|_{U_i}\), and hence of \(f|_{P_i}\), so that  \(\|D\phi_i\|\leq \lambda<1\) for all \(i\in\{1,2,\dots,m\}\). Denote a $\cC^3$ IFS associated to \(f\) as 
		 \[
		 \Phi_{\! f}=\{\phi_1,\phi_2,\ldots,\phi_m\}.
		 \] 
\begin{rem}\label{rem:finite-to-one} Unfortunately, it is not obvious if one can obtain an extension in which $\phi_i$ are invertible. This would simplify the following arguments as one would not have to struggle to restrict the discussion to the sets $U_i$ (e.g. see \eqref{eq:admissible}). However, since $\cC^\infty(\bR^d,\bR^d)$ finite to one maps are generic by Tougeron's Theorem, see \cite[Theorem 2.6 page 169]{GG}, we can assume, by an arbitrarily small perturbation of $f$, that the $\phi_i$ are finite to one.
\end{rem}

	 For \(m,n\in\bN\), let \(\Sigma^m_n=\{1,2,\ldots,m\}^n\) and \(\Sigma^m=\{1,2,\ldots,m\}^{\bN}\) be the standard symbolic spaces. We endow \(\Sigma^m\) with the metric \(d_\gamma\), for some \(\gamma >1\), 
\begin{equation}\label{eq:gamma}
		 d_{\gamma}(\sigma, \sigma^{\prime})=\sum_{i=1}^{\infty}\frac{|\sigma_i-\sigma^{\prime}_i|}{\gamma^i}.
\end{equation}
In addition, let \(\tau:\Sigma^m\to \Sigma^m\) be the left subshift: \(\tau(\sigma_1,\sigma_2,\sigma_3, \ldots)= (\sigma_2,\sigma_3,\ldots)\).\\
	 Set\footnote{\(\norm{\cdot}\) is the general Euclidean norm.} $K=\max\{\|x\|\;:\;x\in X\}$. Let \(M=\sup_i\|\phi_i(0)\|\) then for each $y\in\bR^d$, 
		 \[
		 \norm{\phi_i(y)}\leq \norm{\phi_i(y)-\phi_i(0)}+M\leq \lambda \norm{y}+M .
		 \]
	 Thus, setting 
\begin{equation}\label{eq:Y}
Y=\{y\in \bR^d\;:\; \|y\|\leq \max\{K,(1-\lambda)^{-1} M\}\},
\end{equation}
we have \(\phi_i(Y)\subset Y\), for all $i\in\{1,\dots, m\}$, \(X\subset Y \) and \(Y\) is a \(d\)-dimensional manifold with boundary.\\
	  Next, define \(\Theta_{\! f}:\Sigma^m\to Y\) as
 \begin{equation}\label{eq:defn_psi}
		  \Theta_{\! f}(\sigma)=\bigcap\limits_{n\in\bN} {\phi_{\sigma_1}\circ \phi_{\sigma_2}\circ \ldots \circ \phi_{\sigma_n}(Y)}.
 \end{equation}
     where \(\sigma=(\sigma_1, \sigma_2,\ldots,)\in \Sigma^m\).
	 The sets \(\{\phi_{\sigma_1}\circ \phi_{\sigma_2}\circ \ldots \phi_{\sigma_n}(Y)\}_{n\in\bN}\) form a nested sequence of compact subsets of  \(Y\). Additionally, \(\operatorname{diam}({\phi_{\sigma_1}\circ \phi_{\sigma_2}\circ \cdots\circ \phi_{\sigma_n}(Y)})\to 0\) as \(n\to\infty\), so by Cantor's Intersection theorem \(\Theta_{\! f}(\sigma)\) is a single element in \(Y\) for every \(\sigma \in \Sigma\) which implies \(\Theta_{\! f}\) is well defined. 
	 We define the attractor of the IFS \(\Phi_{\! f}\) as
 \begin{equation}\label{eq:att}
		 \Lambda(\Phi_{\! f})=\Theta_{\! f}(\Sigma^m).
 \end{equation}
 \begin{lem}\label{lem:compact_attractor}
	 The function \(\Theta_{\! f}:\Sigma^m\to Y\) is continuous. In turn, \(\Lambda(\Phi_{\! f})\) is compact.
 \end{lem}
 \begin{proof}
	 For given \(\ve>0\), there exists \(k\in\bN\) such that \(\lambda^k\operatorname{diam}(Y)<\ve\), let \(\delta=\gamma^{-k}\), where $\gamma$ is the one in \eqref{eq:gamma}. For \(\sigma,\sigma^\prime\in \Sigma^m\), \(d_{\gamma}(\sigma, \sigma^{\prime})<\delta=\gamma^{-k}\) implies, for all \(i<k\), \(\sigma_i=\sigma^{\prime}_i\), hence, for any \(x,y\in Y\),
 \begin{align*}
		& d_0(\Theta_{\! f}(\sigma), \Theta_{\! f}(\sigma^\prime)) \leq\\
		&\leq\sup_{x,y\in Y} d_0(\phi_{\sigma_1}\circ\cdots\circ\phi_{\sigma_k}\circ\phi_{\sigma_{k+1}}\ldots\circ\phi_{\sigma_n}(x),\phi_{\sigma_1}\circ\cdots\circ\phi_{\sigma_k}\circ\phi_{\sigma^{\prime}_{k+1}}\ldots\circ \phi_{\sigma^{\prime}_n}(y))\\
		 &<\lambda^k \sup_{x,y\in Y} d_0(\phi_{\sigma_{k+1}}\circ\cdots\circ\phi_{\sigma_n}(x),\phi_{\sigma^{\prime}_{k+1}}\ldots\circ \phi_{\sigma^{\prime}_n}(y))\\
		 &<\lambda^k \operatorname{diam}(Y)<\ve,
 \end{align*}
hence the continuity with respect to \(\sigma\). 
     Since \(\Sigma^m\) is compact, the attractor \(\Lambda(\Phi_{\! f})\) is compact being the continuous image of a compact set.
	 \end{proof}
\begin{rem}
	 Let \(\Phi_{\! f}=\{\phi_1,\ldots,\phi_m\}\) be an IFS associated to a piecewise contraction \(f\). For \(p\in\bN\), let \(\bsP(f^p)=\{P^p_1,P^p_2,\ldots, P^p_{m_p}\}\) be the partition of \(f^p\) given as in equation~(\ref{eq:n-partn}). Define the corresponding IFS associated to \(f^p\) as
		 \[
		 \Phi_{f^p}=\{\varphi_1,\varphi_2,\ldots,\varphi_{m_p}\}
		 \]
	 where for all \(i\in\{1,2,\ldots,m_p\}\) there exists unique  \(\sigma^i=(\sigma^i_1,\ldots,\sigma^i_p)\in\Sigma_p^m\) (uniquely determined by the partition element \(P_i^p\in\bsP(f^p)\)) such that 
	 \[
	 \varphi_i=\phi_{\sigma_1^i}\circ \phi_{\sigma_2^i}\circ\cdots\circ\phi_{\sigma_p^i}.
	 \]
 The attractor of \(\Phi_{f^p}\) is \(\Lambda(\Phi_{f^p})= \Theta_{\! f^p}(\Sigma^{m_p})\). 
	 To avoid confusion we denote the elements in \(\Sigma^m\) by \(\sigma\) and the elements in \(\Sigma^{m_p}\) by \(\omega\). 
\end{rem}
 \begin{lem} \label{lem:att-IFS}
	 For a piecewise smooth contraction \(f\) with IFS \(\Phi_{\! f}\), and \(p\in\bN\), the following relation holds: 
		 \[ 
		 \Lambda(f)\subset\Lambda(f^p)\subset  \Lambda(\Phi_{f^p})\subset \Lambda(\Phi_{\! f}).
		 \] 
 \end{lem}
 \begin{proof}
	 For \(p\in\bN\), let \(\bsP(f^p)=\{P_1^p,P_2^p,\ldots,P_{m_p}^p\}\), and \(m_1=m\).
	 We start by proving the first inclusion, that is \(\Lambda(f)\subset \Lambda(f^p)\).  Let \(x\in \Lambda(f)= \cap_{n\in\bN}\overline{{f}^n(X)}\), that is, for every \(n\in\bN\), \(x\in \overline{f^n(X)}\). Accordingly, 
	 \[
	 x\in \bigcap_{n\in\bN}\overline{(f^p)^n(X)}= \Lambda(f^p).
	 \]
	 Thus \(\Lambda(f)\subset \Lambda(f^p)\). For the third inclusion, that is, \(\Lambda(\Phi_{f^p})\subset \Lambda(\Phi_{\! f})\): let \(\Phi_{f^p}=\{\vf_1,\vf_2,\ldots,\vf_{m_p}\}\), then for every \(\omega=(\omega_1,\omega_2,\ldots)\in\Sigma^{m_p}\) there exists  \(\sigma^{\omega}=(\sigma^{\omega_1},\sigma^{\omega_2}, \dots)\), where \(\sigma^{\omega_i}=(\sigma^{\omega_i}_1,\ldots,\sigma^{\omega_i}_p)\in\Sigma^{m}_p\), such that 
	 \[
\Theta_{\! f^p}(\omega)=\bigcap_{n\in\bN}\varphi_{\omega_1}\circ\varphi_{\omega_2}\circ\cdots\circ\varphi_{\omega_n}(Y)=\bigcap_{n\in\bN} \phi_{\sigma^{\omega_1}_1}\circ\cdots\circ\phi_{\sigma^{\omega_1}_p}\circ\cdots\circ\phi_{\sigma^{\omega_n}_p}(Y).
	 \] 
Thus,  \(\Lambda(\Phi_{f^p})= \bigcup_{\omega\in\Sigma^{m_p}}\Theta_{\! f^p}(\omega)=\bigcup_{\{\sigma^\omega\;:\;\omega\in\Sigma^{m_p}\}}\Theta_{\! f}(\sigma^\omega)\subset \bigcup_{\sigma\in\Sigma^{m}}\Theta_{\! f}(\sigma)=\Lambda(\Phi_{\! f}) \). Finally, for the second inclusion: let \(x\in \Lambda(f^p)=\bigcap_{n\in\bN}\overline{f^{pn}(X)}\) then for all \(n\in\bN\) there exists \(y_n\in X\) such that \(d_0(x,f^{pn}(y_n))<1/n\). By definition of \(\Phi_{f^p}\), there exists \(\omega^n=(\omega_1^n,\omega_2^n,\ldots,\omega_n^n,\ldots)\in\Sigma^{m_p}\) such that \(f^{pn}(y_n)=\vf_{\omega_1^n}\circ\cdots\circ \vf_{\omega_n^n}(y_n)\) where \(\vf_{\omega_i^n}\in\Phi_{f^p}\).  This implies, for all \(n\in\bN\), \(d_0(x,\vf_{\omega_1^n}\circ\cdots\circ \vf_{\omega_n^n}(y_n))<1/n\).
	 By compactness of \(\Sigma^{m_p}\) there exists a subsequence $\{n_k\}$ and \(\omega\in\Sigma^{m_p}\) such that  \(\omega^{n_k}\to\omega\). Since, by Lemma \ref{lem:compact_attractor}, \(\Theta_{\! f^p}\) is continuous we have
	     \[ 
	    d_0(x, \Theta_{\! f^p}(\omega))= \lim_{k\to\infty} d_0(x, \Theta_{\! f^p}(\omega^{n_k}))=\lim_{k\to\infty} d_0(x,f^{pn_k}(y_{n_k}))=0. 
		 \]
Hence,  by definition of the attractor \(x\in \Lambda(\Phi_{f^p})\).
 \end{proof}
 \subsection{A simple perturbation and the Proof of Proposition \ref{prop:markov-dense}}\ \\
 \label{sec:4.2}
 For \(\delta\in\bR^d\) with \(|\delta|>0\) sufficiently small, and a piecewise contraction \(f\) with IFS \(\Phi_{\! f}=\{\phi_1,\phi_2,\ldots,\phi_m\}\), we define perturbations \(f^{\delta}, \Phi_{f^\delta}\) as follows:
 \begin{equation}\label{eq:defn_pert}
     f^{\delta}(x)=f(x)+\delta,\ \phi_i^{\delta}=\phi_i+\delta.
 \end{equation}
Provided \(\abs{\delta}\) is small enough, the perturbation \(f^\delta\) satisfies \(\overline{f^\delta(X)}\subset\mathring X\), hence \(f^\delta\) is a piecewise smooth contraction with corresponding IFS \(\Phi_{f^\delta}=\{\phi^\delta_1,\phi^\delta_2,\ldots,\phi^\delta_m\}\). One can easily check that \(d_2(f,f^\delta)=|\delta|\). By definition, see (\ref{eq:defn_psi}), \(\Theta_{\! f^\delta}:\Sigma^m\to Y\) reads 
         \[
	     \Theta_{\! f^{\delta}}(\sigma)= \bigcap_{n \in\bN}{\phi_{\sigma_1}^{\delta}\circ\phi_{\sigma_2}^{\delta}\circ\cdots\circ\phi_{\sigma_n}^{\delta}(Y)}
		 \]
	 and the respective attractor is \({\Lambda}(\Phi_{f^\delta})= \bigcup_{\sigma\in\Sigma^m}\Theta_{\! f^{\delta}}(\sigma)\).\\
     Observe that for any \(p\in\bN\), the corresponding IFS associated to \((f^{\delta})^p\) is given by \(\Phi_{(f^\delta)^p}=\{\varphi_1^\delta, \varphi_2^\delta, \ldots, \varphi_{m_p}^\delta\}\) where for every \(i\in\{1,2,\ldots,m_p\}\) there exists \(\sigma_i\in\Sigma^m_p\) such that \(\varphi_i^\delta= \phi_{\sigma_1^i}^\delta\circ \phi_{\sigma_2^i}^\delta\circ \ldots\circ \phi_{\sigma_p^i}^\delta\).
 \begin{lem}
The map	 \(\Theta_{\! f^{\delta}}(\sigma)\mapsto \bigcap_{n \in\bN}{\phi_{\sigma_1}^{\delta}\circ\phi_{\sigma_2}^{\delta}\circ\cdots\circ\phi_{\sigma_n}^{\delta}(Y)}\) is uniformly Lipschitz continuous in \(\delta\), that is, there exists \(a>0\) such that, for all $\sigma\in\Sigma^m$, \(d_0(\Theta_{\! f^{\delta}}(\sigma),\Theta_{\! f^{\delta^{\prime}}}(\sigma))\leq a d_0(\delta,\delta^{\prime})\).
 \end{lem}
 \begin{proof}
	 Let \(\delta,\delta^\prime>0\), $n\in\bN$, \(x\in Y\), and  \(\sigma=(\sigma_1,\sigma_2,\ldots,\sigma_n,\ldots)\), then
 \begin{align*}
	 d_0(\Theta_{\! f^{\delta}}(\sigma), &\Theta_{\! f^{\delta^\prime}}(\sigma))\leq d_0(\Theta_{\! f^{\delta}}(\sigma),\phi_{\sigma_1}^{\delta}\circ \ldots \circ \phi_{\sigma_n}^{\delta}(x) )+d_0(\Theta_{\! f^{\delta^\prime}}(\sigma),\phi_{\sigma_1}^{\delta^{\prime}}\circ \ldots \circ \phi_{\sigma_n}^{\delta^{\prime}}(x))\\
	 &+d_0(\phi_{\sigma_1}^{\delta}\circ \ldots \circ \phi_{\sigma_n}^{\delta}(x),\phi_{\sigma_1}^{\delta^{\prime}}\circ \ldots \circ \phi_{\sigma_n}^{\delta^{\prime}}(x))\\
	 &\leq \lambda^nd_0(\Theta_{\! f^{\delta}}(\tau^n\sigma),x)+ \lambda^n d_0(\Theta_{\! f^{\delta^\prime}}(\tau^n\sigma),x)\\
	 &+d_0(\delta,\delta^{\prime})
	 +d_0(\phi_{\sigma_1}\circ \phi_{\sigma_2}^{\delta}\circ \ldots \circ \phi_{\sigma_n}^{\delta}(x),\phi_{\sigma_1}\circ \phi_{\sigma_2}^{\delta^\prime}\circ \ldots \circ \phi_{\sigma_n}^{\delta^{\prime}}(x))\\
	 &\leq 2\lambda^n\operatorname{diam}(Y) +d_0(\delta,\delta^{\prime})+\lambda d_0(\phi_{\sigma_2}^{\delta}\circ \ldots \circ \phi_{\sigma_n}^{\delta}(x),\phi_{\sigma_2}^{\delta^\prime}\circ \ldots \circ \phi_{\sigma_n}^{\delta^{\prime}}(x)).
 \end{align*}
Iterating the above argument yields
\[
\begin{split}
 d_0(\Theta_{\! f^{\delta}}(\sigma), \Theta_{\! f^{\delta^\prime}}(\sigma))
 &\leq \lim_{n\to\infty}\left\{2\lambda^n\operatorname{diam}(Y)+ d_0(\delta,\delta^{\prime})(1+\lambda+\lambda^2+\ldots+ \lambda^n)\right\}\\
 &=(1-\lambda)^{-1}d_0(\delta,\delta^{\prime}),
\end{split}
\] 
letting \(a=1/(1-\lambda)\) concludes the proof.
 \end{proof}
 \begin{proof}[\bf \em Proof of Proposition~\ref{prop:markov-dense}]
Let $p\in\bN$ be such that $m_p\lambda^p\leq \frac 12$. 
Lemma~\ref{lem:att-IFS} asserts that \(\Lambda(f^\delta)\subset \Lambda(\Phi_{(f^\delta)^p})\). Hence, by Theorem \ref{thm:eq}, it suffices to prove that, for every \(\ve>0\) small enough, there exists \(\delta\in B_{\ve}(0)\) such that the attractor \(\Lambda(\Phi_{(f^\delta)^p})\) is disjoint from \(\partial \bsP(f^\delta)=\partial \bsP(f)\).\\
Suppose to the contrary that for every \(\delta\in B_{\ve}(0)\), \(\Lambda(\Phi_{(f^\delta)^p})\cap\partial\bsP(f)\neq \emptyset\). Accordingly, there exists \(\omega(\delta)\in \Sigma^{m_p}\) for which \(\Theta_{\! (f^\delta)^{p}}(\omega(\delta))\in \partial\bsP(f)\).
By definition, \(\partial \bsP(f)=\cup_{P\in\bsP(f)}\partial P\), therefore there exists \(P_i\in\bsP(f)\) and \(A\subset B_{\ve}(0)\) with\footnote{\(\mu_d\) is the \(d\)-dimensional Lebesgue measure,} \(\mu_d(A)\geq \mu_d(B_{\ve}(0))/m = C_d\ve^d/m\),  such that, for all \(\delta \in A,\  \Theta_{\! (f^\delta)^p}(\omega(\delta))\in \partial P_i \). \\
Moreover, for each $k\in\bN$, there exist \(\omega^*= (\omega^*_1, \omega^*_2,\ldots,\omega^*_k)\in \Sigma^{m_p}_k\) such that the set defined as 
\[
A_k(\omega^*)= \{\delta\in A: \omega(\delta)_j=\omega^*_j,j\leq k\}
\]
		 is non empty and \(\mu_d(A_k(\omega^*))\geq \mu_d(A)/m_p^k\geq C_d\ve^dm^{-1}{m_p}^{-k}\). Accordingly, for \(\omega(\delta)\in A_k(\omega^*)\) and IFS associated to \(\Phi_{(f^\delta)^p}=\{\vf_1^\delta,\vf_2^\delta,\ldots,\vf_{m_p}^\delta\}\),
 \begin{align*}
		 \partial P_i\ni \Theta_{\! (f^\delta)^p}(\omega(\delta))&=\varphi_{\omega^*_1}^{\delta}\circ\Theta_{\! (f^\delta)^p}(\tau \omega(\delta))= \varphi_{\omega^*_1}^{\delta}\circ\varphi_{\omega^*_2}^{\delta}\circ\cdots \circ\varphi_{\omega^*_k}^{\delta}\circ\Theta_{\! (f^\delta)^p}(\tau^k\omega(\delta))
 \end{align*}
		 where \(\tau\) is the left shift as defined above, and	 \(\varphi_{\omega_j^*}^\delta= \phi^\delta_{\sigma^p_{j,1}}\circ \phi_{\sigma^p_{j,2}}^\delta\circ\cdots\circ \phi_{\sigma^p_{j,p}}^\delta\), for  \(\omega^*_j=(\sigma^p_{j,1},\sigma^p_{j,2},\ldots,\sigma^p_{j,p})\in \Sigma^m_p\) with \(\phi^\delta_{\sigma^p_{j,s}}\in \Phi_{f^\delta}\).\\
	Next, for some \(\bar{x}\in X\), define \(\theta:A_k(\omega^*)\to X\) as \(\theta(\delta)=\varphi_{\omega^*_1}^{\delta}\circ\varphi_{\omega^*_2}^{\delta}\circ\cdots\circ \varphi_{\omega^*_k}^{\delta}(\bar{x})\) then
	\[
		 d_0(\Theta_{\! (f^\delta)^p}(\omega(\delta)),\theta(\delta))\leq\lambda^{pk}d_0(\Theta_{\! (f^\delta)^p}(\tau^k\omega(\delta)),\bar{x})\leq \operatorname{diam}(Y)\lambda^{pk}=:B_*\lambda^{pk}.
	\]
	Thus, \(\theta(\delta)\) belongs to a \(B_*\lambda^{pk}\) neighbourhood of \(\partial P_i\). Since $\partial P_i$ is contained in the union of finitely many $\cC^2$ manifolds,  the Lebesgue measure of a \(B_*\lambda^{pk}\) neighbourhood of \(\partial P_i\)  is bounded above by \(C\lambda^{pk} \mu_{d-1}(\partial P_i)\) for a fixed constant \(C>0\). Accordingly,
 \begin{equation}\label{eq:theta1}
	\mu_d(\theta(A_k(\omega^*)))\leq C\lambda^{pk} \mu_{d-1}(\partial P_i).
 \end{equation}
	On the other hand, 
			 \[
			 \mu_d(\theta(A_k(\omega^*)))=\int_{A_k(\omega^*)}|\det(D\theta(\delta))|d\delta
			 \] 
		 where by definition of \(\theta(\delta)\),  \(D\theta(\delta)=\Id+D\varphi_{\sigma_1^*}+D\varphi_{\sigma_1^*}D\varphi_{\sigma_2^*}+\cdots+D\varphi_{\sigma_1^*}\cdots D\varphi_{\sigma_{k-1}^*}\). Note that
	\[
		 \norm{D\varphi_{\sigma_1^*}+D\varphi_{\sigma_1^*}D\varphi_{\sigma_2^*}+\cdots+D\varphi_{\sigma_1^*}\cdots D\varphi_{\sigma_{k-1}^*}}\leq \frac{\lambda^p}{1-\lambda^p}
	\]
	where \(\norm{\cdot}\) is the standard operator norm defined as \(\norm{L}=\sup_{\|v\|=1}\norm{Lv}\) for any linear operator \(L:\bR^d\to\bR^d\). It follows that, for all $v\in\bR^d$, 
	\[
	\|D\theta(\delta)v\|\geq \|v\|-\frac{\lambda^p}{1-\lambda^p}\|v\|\geq\frac{1}{2}\|v\|
	\]
	since  \(\lambda^p\leq \frac{1}{2m^p}\) and \(m^p\geq2\). Hence the eigenvalues of $D\theta(\delta)$ are larger, in modulus, than $\frac 12$. Accordingly,
	 \(\left|\det(D\theta(\delta))\right|\geq 2^{-d}\)  and 
 \begin{equation}\label{eq:theta2}
			 \mu_d(\theta(A_k(\omega^*)))\geq 2^{-d} \mu_d(A_k(\omega^*))\geq C_d 2^{-d}\ve^dm^{-1}{m_p}^{-k}.
 \end{equation}
	which, for $k$ large enough, is in contradiction with \eqref{eq:theta1}, concluding the proof.
 \end{proof}
\subsection{Fixed points in a generic position}\label{sec:fixed}\ \\

Fix $N\in\bN$. For all $q\leq N+1$ and $\sigma=(\sigma_1,\ldots,\sigma_q)\in \Sigma^m_{q}(\Phi)$, let $x_\sigma(\Phi)$ be the unique fixed point of $\phi_{\sigma_1}\circ\cdots\circ \phi_{\sigma_q}$, that is
 \begin{equation}\label{eq:fix-def}
 \phi_{\sigma_1}\circ\cdots\circ \phi_{\sigma_q}(x_\sigma(\Phi))=x_\sigma(\Phi).
 \end{equation}
The goal of this section is to define a perturbation that puts the above fixed points in a generic position. We start with the following trivial but useful fact concerning the location of such fixed points.

 \begin{lem}\label{lem:fix}
Given an IFS $\Phi$, if for some $y\in\bR^d, \delta>0$, $p\in\bN$ and $\sigma\in\Sigma^m_p$, 
\[
\phi_{\sigma_1}\circ\cdots\circ \phi_{\sigma_p}(B_\delta(y))\cap B_\delta(y)\neq \emptyset,
\]
then the unique fixed point of $\phi_{\sigma_1}\circ\cdots\circ \phi_{\sigma_p}$ belongs to $ B_{c_*\delta}(y)$,  $c_*=\frac{2}{1-\lambda}>2$.
\end{lem}
\begin{proof} The fact
 \(\|\phi_{\sigma_1}\circ\cdots\circ \phi_{\sigma_p}(y)-y\|\leq 2\delta\)
 implies that, for each $x\in B_{c_*\delta}(y)$, 
 \[
 \begin{split}
 \|\phi_{\sigma_1}\circ\cdots\circ \phi_{\sigma_p}(x)-y\|&\leq \|\phi_{\sigma_1}\circ\cdots\circ \phi_{\sigma_p}(x)-\phi_{\sigma_1}\circ\cdots\circ \phi_{\sigma_p}(y)\|+2\delta\\
 &\leq (\lambda^p c_*+2)\delta\leq c_*\delta.
 \end{split}
 \]
 Hence  $\phi_{\sigma_1}\circ\cdots\circ \phi_{\sigma_p}(B_{c_*\delta}(y))\subset B_{c_*\delta}(y)$. The Lemma follows by the contraction mapping theorem.
 \end{proof}
Let $\cA_p=\{\phi_{\sigma_{1}}\circ\cdots\circ \phi_{\sigma_p}\;:\;\sigma\in \Sigma^{m}_{p}\}$ and $A_p(\Phi):=\{x_\sigma(\Phi)\;:\; \sigma\in  \Sigma^m_{p}(\Phi)\}$ and let \(\#A_{N}(\Phi)=k_\star\). We can now explain what we mean by having the fixed points in a generic position.
 \begin{prop} \label{prop:fix-sep}
 For each $\ve>0$ there exists an $\ve$-perturbation $\Phi^0=\{\phi^0_k\}$ of $\Phi$ such that for all $q\leq p\leq N$, $\sigma\in  \Sigma^m_{q}(\Phi)$, $\omega\in  \Sigma^m_{p}(\Phi)$, $\omega\neq \sigma$, if $\sigma_{p+1}\neq \omega_q$ then $x_{\sigma}(\Phi^0)\neq x_{\omega}(\Phi^0)$, while if $\sigma_{p+1}=\omega_q$ and $x_{\sigma}(\Phi^0)= x_{\omega}(\Phi^0)$ then $\phi_{\sigma_{1}}\circ\cdots\circ \phi_{\sigma_q}$ and $\phi_{\omega_{1}}\circ\cdots\circ \phi_{\omega_p}$ are both some power of a $\Theta\in\bigcup_{s=1}^p \cA_s$.
 \end{prop}
 \begin{proof}
 We proceed by induction on $p$. If $\sigma_1,\omega_1\in \{1,\dots,m\}$ and $x_{\sigma_1}(\Phi)=x_{\omega_1}(\Phi)$, then we can simply make the perturbation $\tilde \phi_{\sigma_1}(x)=\phi_{\sigma_1}(x)+\eta$ for some $\|\eta\|<\ve/2$. This proves the statement for $p=1$. To ease the notation we keep calling $\Phi$ the perturbed IFS. Let us suppose the statement is true for $p$, after a perturbation of size at most $(1-2^{-p})\ve$, and prove it for $p+1$. \\
Since $A(p+1)=\{x_\sigma(\Phi)\;:\; \sigma\in  \Sigma^m_{p+1,*}(\Phi)\}$ is a finite discrete set we have 
\[
\delta^*_{p+1}=\min\left\{1,\inf_{\substack{x,y\in A(p+1)\\ x\neq y}}\|x-y\|\right\}>0.
\]
Let $\delta\leq c_*^{-2}\delta^*_{p+1}/2$.\footnote{Recall that $c_*=\frac{2}{1-\lambda}$.}
 Suppose that  $z:=x_\sigma(\Phi)=x_\omega(\Phi)$ and $\omega_{q}=\sigma_{p+1}$, where $\sigma\in \Sigma^m_{p+1}(\Phi)$, $\omega\in \Sigma^m_{q}(\Phi)$, $q\leq p+1$ and $\sigma\neq \omega$. Let $j\in\{0,\dots, q-1\}$ be the largest integer such that $\sigma_{p+1-j}=\omega_{q-j}$. If $j=q-1$, then it must be $q\leq p$, otherwise we would have $\omega=\sigma$.
 Hence
 \[
\phi_{\omega_{1}}\circ \cdots\circ\phi_{\omega_q}(z)=z=\phi_{\sigma_1}\circ \cdots \circ\phi_{\sigma_{p+1}}(z)=\phi_{\sigma_1}\circ \cdots \circ\phi_{\sigma_{p+1-q}}\circ \phi_{\omega_{1}}\circ \cdots\circ\phi_{\omega_q}(z).
 \]
 That is $\phi_{\sigma_1}\circ \cdots \circ\phi_{\sigma_{p+1-q}}(z)=z$. It follows by the inductive hypothesis that there exists $k\in\bN$ and $\Theta\in \cA_s$, $s\leq q$, such that
 \[
 \begin{split}
& \phi_{\sigma_{1}}\circ \cdots \circ\phi_{\sigma_{p+1-q}}=\Theta^k\\
&\phi_{\omega_{1}}\circ \cdots\circ\phi_{\omega_q}=\Theta^j
\end{split}
\]
 so, $\phi_{\sigma_1}\circ \cdots\circ\phi_{\sigma_{p+1}}=\Theta^{k+j}$, as claimed. It remains to consider the case $j<q-1$.
Let $\psi:= \phi_{\omega_{q-j}}\circ \cdots\circ\phi_{\omega_q}$, then we have $\psi=\phi_{\sigma_{p+1-j}}\circ\cdots\circ \phi_{\sigma_{p+1}}$, and
 \[
\left[\psi\circ \phi_{\omega_1}\circ \cdots \circ \phi_{\omega_{q-j-1}}\right]( \psi(z))=\psi(z)= \left[\psi\circ \phi_{\sigma_1}\circ \cdots \circ\phi_{\sigma_{p-j}}\right](\psi(z)),
 \]
 where, by construction, $\omega_{q-j-1}\neq \sigma_{p-j}$. By renaming the indices, we are thus reduced to the case $\omega_{q}\neq \sigma_{p+1}$. The following lemma is useful to analyse this case.
\begin{sublem}\label{sublem:uff}
If, for $j\in\{1,\dots, p\}$,
 \[
\phi_{\sigma_{j+1}}\circ\cdots \circ \phi_{\sigma_{p+1}}(B_\delta(z))\cap B_\delta(z)\neq \emptyset,
 \]
 then $\phi_{\sigma_{j+1}}\circ\cdots \circ \phi_{\sigma_{p+1}}(z)=z$ and $\sigma_j=\sigma_{p+1}$, and the same for $\omega$.
\end{sublem}
\begin{proof}
Lemma \ref{lem:fix} implies that there exists $z_1\in B_{c_*\delta}(z)$ such that $\phi_{\sigma_{j+1}}\circ\cdots \circ \phi_{\sigma_{p+1}}(z_1)=z_1$. In addition,
\[
\|z-\phi_{\sigma_{1}}\circ\cdots \circ \phi_{\sigma_{j}}(z_1)\|=\|\phi_{\sigma_{1}}\circ\cdots \circ \phi_{\sigma_{p+1}}(z)-\phi_{\sigma_{1}}\circ\cdots \circ \phi_{\sigma_{p+1}}(z_1)\|\leq\lambda c_*\delta.
\]
Thus, $\phi_{\sigma_{1}}\circ\cdots \circ \phi_{\sigma_{j}}(B_{c_*\delta}(z))\cap B_{c_*\delta}(z)\neq \emptyset$, hence  Lemma \ref{lem:fix} implies that there exists $z_2\in B_{c_*^2\delta}(z)$ such that 
$\phi_{\sigma_{1}}\circ\cdots \circ \phi_{\sigma_{j}}(z_2)=z_2$. The definition of $\delta$ implies that $z_1=z_2$, which in turn implies $z_1=z$. Hence, we have 
\[
\begin{split}
&\phi_{\sigma_{1}}\circ\cdots \circ \phi_{\sigma_{j}}(z)=z\\
&\phi_{\sigma_{j+1}}\circ\cdots \circ \phi_{\sigma_{p+1}}(z)=z
\end{split}
\]
and, by the inductive hypothesis, this is possible only if $\sigma_{j}=\sigma_{p+1}$. The argument for $\omega$ is identical.
\end{proof}
If there exists $j\in\{1,\dots, p\}$ and $k\in\{1,\dots, q\}$ such that
\begin{equation}\label{eq:repeat}
\begin{split}
&\phi_{\sigma_{j+1}}\circ\cdots \circ \phi_{\sigma_{p+1}}(B_\delta(z))\cap B_\delta(z)\neq \emptyset\\
&\phi_{\omega_{k+1}}\circ\cdots \circ \phi_{\omega_{q}}(B_\delta(z))\cap B_\delta(z)\neq \emptyset
\end{split}
\end{equation}
then, by Sub-Lemma \ref{sublem:uff}, $\phi_{\sigma_{j+1}}\circ\cdots \circ \phi_{\sigma_{p+1}}(z)=z=\phi_{\omega_{k+1}}\circ\cdots \circ \phi_{\omega_{q}}(z)$ which contradicts our inductive hypothesis. Thus, if the first inequality is satisfied for some $j$, the second cannot be satisfied for any $k$ and vice versa. Let us suppose that there does not exist $k$ for which the second inequality of \eqref{eq:repeat} is satisfied (the other possibility being completely analogous).

Define the perturbation
\[
\tilde\phi_k=\begin{cases}
\phi_k\circ h_{z,\delta} &\textrm{ if } k=\omega_{q}\\
\phi_k&\textrm{ otherwise.} 
\end{cases}
\]
where, in analogy with \eqref{eq:hdef}, for $v\in\bR^d$, $\|v\|=1$,
\[
			h_{z,\delta}(x)=\begin{cases}
			x\quad&\forall x\not\in B_{c_*^{-1}\delta}(z)\\
			x+c_*^{-3}\delta^3g(1-c_*\delta^{-1}\|x-z\|) v   &\textrm{ otherwise},
			\end{cases}
\]
Note that $\|h_{z,\delta}-id\|_{\cC^2}\leq 2^{-p-1}\ve_0$. Since $h_{z,\delta}(B_{c_*^{-1}\delta}(z))\subset B_{c_*^{-1}\delta}(z)$ it follows that the effect of the perturbation is always confined to $B_{\delta}(z)$ and its images. Moreover, by Sub-Lemma \ref{sublem:uff}, if there exists $j$ such that $\phi_{\sigma_{j+1}}\circ\cdots \circ \phi_{\sigma_{p+1}}(B_\delta(z))\cap B_\delta(z)\neq\emptyset$, we have seen that $\sigma_{j}=\sigma_{p+1}\neq \omega_q$, so next we apply a map, $\phi_{\sigma_j}$, that has not been modified. On the other hand, if $x\in B_\delta(z)$ and for some $k$ we have $\omega_k=\omega_{q}$, we have that $\phi_{\omega_{k+1}}\circ\cdots\circ\phi_{\omega_q}(x)\not \in B_\delta(z)$. Next, we will apply $\phi_{\omega_k}$ outside the region where it has been modified. It follows that for each $x\in  B_\delta(z)$ we have
\[
\begin{split}
& \tilde\phi_{\omega_{1}}\circ\cdots \circ \tilde\phi_{\omega_{q}}(x)=\phi_{\omega_{1}}\circ\cdots \circ \phi_{\omega_{q}}\circ h_{z,\delta}(x)\\
&\tilde \phi_{\sigma_{1}}\circ\cdots \circ \tilde \phi_{\sigma_{p+1}}(x)=\phi_{\sigma_{1}}\circ\cdots \circ \phi_{\sigma_{p+1}}(x).
\end{split}
\]
Calling $z(\delta)$ the unique fixed point of $\phi_{\omega_{1}}\circ\cdots \circ \phi_{\omega_{q}}\circ h_{z,\delta}$ we have
\[
z'(\delta)=(\Id-D\phi_{\omega_{1}}\cdots D \phi_{\omega_{q}}Dh_{z,v})^{-1}D\phi_{\omega_{1}}\cdots D \phi_{\omega_{q}}\partial_\delta h_{z,\delta}.
\]
Since $z(0)=z$ and
\[
\partial_\delta h_{z,\delta}=3c_*^{-3}\delta^2g(1-c_*\delta^{-1}\|z(\delta)-z\|) v +c_*^{-2}\delta \|z(\delta)-z\|g'(1-c_*\delta^{-1}\|z(\delta)-z\|) v
\]
we have that it is not possible that $z(\delta)=z$ for all $\delta \leq c_*^{-2}\delta^*_{p+1}/2$. Thus we can make a perturbation for which the two fixed points are different, and, for $\delta$ small, they cannot be equal to the other fixed points.

For any other couple of elements \(\sigma\in\Sigma^m_{p+1}(\Phi),\omega\in\Sigma^m_{q}(\Phi)\), we can repeat the same process and obtain the perturbation with two different fixed points as above. Note that the size of perturbation being \(\delta<c_*^{-1}\delta^*_{p+1}/2\), the distance between the newly obtained fixed points in \(\cup_{q\leq p}\cA_q\) stays positive as the perturbation doesn't move the fixed points more than $\delta^*_{p+1}/2$.
 \end{proof}
 
From now on, we assume that $\Phi$ satisfies Proposition \ref{prop:fix-sep}.
\subsection{Pre-images of the boundary manifolds and how to avoid them}\label{sec:manifold}\ \\
Next, we need several notations and a few lemmata to describe the structure of the pre-images of the discontinuity manifolds conveniently. This will allow us to develop the tools to prove Proposition~\ref{prop:psc-dense}. \\
Let $f$ be a piecewise smooth contraction with \(\bsP(f)=\{P_1, P_2,\ldots, P_m\}\) and \(\Phi_{\! f}=\{\phi_1,\phi_2,\ldots,\phi_m\}\).
Recall that, by hypothesis, $\partial \bsP(f)$ is contained in the finite union of $\cC^2$ manifolds, which we will call {\em boundary manifolds}. Let $l_0$ be the number of boundary manifolds in $\partial \bsP(f)$.
Recall also that, for every \(i\in\{1,2,\ldots,m\}\), \(U_{i}\) is the open neighbourhood of \(P_{i}\in\bsP(f)\) such that \(\tilde f|_{U_{i}}\) is injective and hence invertible. Accordingly, by the construction of \(\Phi_{\! f}\), we have that \(\phi_{i}|_{U_{i}}\) has a well defined inverse for all \(i\in\{1,2,\ldots,m\}\). \\
Let $\epsilon_0=\min\{ d_H(P_i,U_i^c)\;:\; i\in\{1,\dots,m\}\}$, where the complement is taken in $\bR^d$. For each $\epsilon\leq \epsilon_0/2$ we can consider the $\epsilon$-neighbourhood $V_i$ of $P_i$ and the $\epsilon/2$ neighbourhood $V_i^-$.\\ 
Choosing $\epsilon$ small enough we can describe the boundary manifolds by embeddings \(\psi_i\in\cC^2(D_i^+, \bR^d)\), 
\(i\in\{1,2,\ldots,l_0\}\), such that  $\overline{\psi_i(D^+_i)}\subset U_{p}$,  for some $p\in\{1,\dots,m\}$, and there exists an open set 
\(D_i\subset D^+_i\subset\bR^{d-1} \) such that $\psi_i(D_i)\cap V_{p}\neq\emptyset$, and $\partial \psi_i(D_i)\cap \overline V_{p}=\emptyset$.\footnote{ This is possible by the Definition \ref{def:psc}.}   
For each IFS $\Phi$ and $\sigma\in \Sigma^m_n$, recalling the definition \eqref{eq:Y} of $Y$, we let
\begin{equation}\label{eq:Ddef}
\begin{split}
&D_{\sigma}(\Phi)=\{x\in Y\;:\; x\in V^-_{\sigma_n}, \phi_{\sigma_{k+1}}\circ\cdots\circ \phi_{\sigma_n}(x)\in V^-_{\sigma_{k}}, 
k\in \{1,\dots, n-1\}\}\\
&D^+_{\sigma}(\Phi)=\{x\in Y: x\in V_{\sigma_n}, \phi_{\sigma_{k+1}}\circ\cdots\circ \phi_{\sigma_n}(x)\in V_{\sigma_{k}}, k\in \{1,\dots, n-1\}\}.
\end{split}
\end{equation}
We call a sequence $\sigma$ {\em admissible} if $D_{\sigma}(\Phi)\neq \emptyset$. We define the set $\Sigma^m_{n,i}$ of the $i$-admissible sequences as
 \begin{equation}\label{eq:admissible}
 \begin{split}
		& \Sigma_{n,i}^m(\Phi)=\big\{\sigma\in\Sigma^{m}_n\;:\; \phi_{\sigma_1}\circ\cdots\circ \phi_{\sigma_n}(D_{\sigma}(\Phi))\cap \psi_i(D_i)\neq\emptyset\big\}\\
	 &D_{\sigma,i}(\Phi)=\{x\in D_{\sigma}(\Phi)\;:\;  \phi_{\sigma_1}\circ\cdots\circ \phi_{\sigma_n}(x)\in\psi_i(D_i)\}.
 \end{split}
 \end{equation}
  \begin{rem}\label{rem:nice_nbd} 
 Note that, for each $N \in\bN$, there is a $\delta>0$ such that, for each $i\in\{1,\dots, l_0\}$, \(n\leq N\), admissible word  \(\sigma\in\Sigma^m_{n,i}\), point  \(x\in D^+_{\sigma}(\Phi_{\! f})\) and small enough perturbations $\tilde\Phi=\{\tilde\phi_i\}$ of  $\Phi$, we have, for each $j\leq n$, \(\tilde \phi_{\sigma_j}\circ\cdots\circ\tilde \phi_{\sigma_n}(B_{\delta}(x))\subset \tilde \phi_{\sigma_{j}}(U_{\sigma_j})\), so that the inverse function $\tilde\phi_{\sigma_n}^{-1}\circ\cdots \circ\tilde \phi_{\sigma_1}^{-1}$ is well defined on \(\tilde\phi_{\sigma_1}\circ\cdots\circ\tilde\phi_{\sigma_n}(B_{\delta}(x))\).
 \end{rem}
 \begin{rem}\label{rem:boundary0} 
 By the Definition~(\ref{eq:n-partn}) of the partition \(\bsP(f^n)\), it follows that  
\[
\partial \bsP(f^n)\subset\bigcup_{i=1}^{l_0}\bigcup_{r=0}^n\bigcup_{\sigma\in\Sigma^m_{r,i}}D_{\sigma,i}(\Phi_{\! f}),
\] 
where $\Sigma^m_{0,i}(\Phi)=\{0\}$ and $\phi_0=id$, so for $\sigma\in\Sigma^m_{0,i}(\Phi)$ we have $D_{\sigma,i}(\Phi_{\! f})=\psi_i(D_i)$.
\end{rem}	
Unfortunately, the sets $D_{\sigma,i}(\Phi_{\! f})$ may have a rather complex topological structure, while we would like to cover $\partial \bsP(f^n)$ with a finite set of $(d-1)$-dimensional manifolds described by a single chart. This is our next task.

In the following, we will write only $\Sigma_{n,i}^m$, if it does not create confusion. In addition, we set 
\begin{equation}\label{eq:sigmazero}
\Sigma_{n,*}^m=\cup_{i=1}^{l_0} \Sigma_{n,i}^m.
\end{equation}
Note that if $\sigma\in\Sigma_{n,*}^m$ and $x\in D_{\sigma,i}(\Phi)$ then there exists a unique $y\in D_i$ such that $\phi_\sigma(x):=\phi_{\sigma_1}\circ\cdots\circ \phi_{\sigma_n}(x)=\psi_i(y)$, so the following is well defined: $\psi_i^{-1}\circ \phi_{\sigma}(D_{\sigma,i}(\Phi))=W_{\sigma,i}$. Note that $W_{\sigma,i}$ is a compact set. For all $y\in D_i$ let $A(y)$ be the set of $\sigma\in \Sigma^m_{p,i}$, $p\leq n$, such that $y\in W_{\sigma,i}$.
As noticed in Remark \ref{rem:nice_nbd}, there exists a $\delta(y)>0$ such that
$\psi_i(B_{\delta(y)}(y))\subset \bigcap_{\sigma\in A(y)}\phi_{\sigma}(D_{\sigma,i}(\Phi))$. For a fixed $N\in\bN$, we have that $\{B_{\delta(y)}(y)\;:\; y\in \overline{D_i}\}$ is a Besicovitch cover and, by Besicovitch covering theorem, we can obtain a subcover in which each point can belong at most at $c_d$ balls (for some $c_d$ depending only on the dimension $d$). We can then extract a finite subcover 
\begin{equation}\label{eq:many-manifolds}
\cW_i^N=\{B_{\delta(y_k)}(y_k)\}
\end{equation}
of $\overline{D_{i}}$. We set, for $\sigma\in\Sigma_{p,i}^{m}$ with $p\in\{0,\dots,N\}$, $M^N_{\sigma,i}(\Phi)=\{\phi_{\sigma_p}^{-1}\circ \cdots\circ\phi_{\sigma_1}^{-1}\circ\psi_i(B_{\delta(y_k)}(y_k))\}$,
$M^N_{0,i}= \{\psi_i(D_i)\}$, this is the wanted collection of $(d-1)$-dimensional manifolds, note that they are not necessarily disjoint. However, they have the wanted property, as the following remark states.
\begin{rem}\label{rem:boundary} 
 By the Definition given by equation~(\ref{eq:n-partn}) of the partition \(\bsP(f^n)\), it follows that, for each $N\geq n$, 
\[
\partial \bsP(f^n)\subset\bigcup_{i=1}^{l_0}\bigcup_{r=0}^n\bigcup_{\sigma\in\Sigma^m_{r,i}}\bigcup_{M\in M^N_{\sigma,i}(\Phi_{\! f})}M,
\] 
where $\Sigma^m_{0,i}=\{0\}$ and $\phi_0=id$, so for $\sigma\in\Sigma^m_{0,i}$ we have $M_{\sigma,i}(\Phi_{\! f})=\psi_i(D_i)$.
\end{rem}

Also, for all IFS $\Phi$ and \(N\in\bN\cup\{0\}\), we define
 \begin{equation}\label{eq:bdset}
			 D^{N}(\Phi)=\bigcup_{n=0}^{N}\bigcup_{i=1}^{l_0}\bigcup_{\sigma\in\Sigma_{n,i}^m}\bigcup_{M\in M^N_{\sigma,i}(\Phi)}  M.
 \end{equation}
Also, for \(\delta>0\), we define  the closure of the \(\delta\)-neighborhood of \(D^{N}(\Phi_{\! f})\) as
 \begin{equation}\label{eq:bdnbd}
		 D_{\delta}^{N}(\Phi)=\bigcup_{x\in D^{N}(\Phi_{\! f})}\overline{B_\delta(x)}.
 \end{equation}

\begin{rem} The basic idea of the proof is to make a perturbation such that the images of \(\partial \bsP(f^n)\) do not self-intersect too many times. This can be easily done for a single map $\phi_i$. However, we are dealing with compositions in which the same map can appear many times. So we have to avoid the possibility that the perturbation at one time interferes with itself at a later time. This can be achieved if there are no fixed points close to the pre-images of the singularities. This is our next task.
\end{rem}

 \begin{prop}\label{prop:bad_points}
	 Let \(\Phi=\{\phi_1,\phi_2,\ldots,\phi_m\}\) be an IFS with contraction coefficient \(\lambda\). For each $N\in\bN$, and $\ve>0$ small enough, there exists an IFS \ \(\tilde \Phi=\{\tilde \phi_1,\ldots,\tilde\phi_m\}\) and \(\delta_\ve\in(0,\ve)\) such that $\|\phi_i-\tilde\phi_i\|_{\cC^2}\leq \ve$, and for any $p\leq N$, \(\sigma=(\sigma_1,\ldots,\sigma_p)\in\Sigma^{m}_{p,*}(\Phi)\) we have that $\tilde\phi_{\sigma_1}\circ\cdots\circ\tilde\phi_{\sigma_p}|_{D_\sigma(\tilde\Phi)}$ is invertible; moreover, \(x\in D_{\delta_\ve}^{N}(\tilde\Phi)\) implies that \(\tilde\phi_{\sigma_1}\circ\cdots\circ\tilde\phi_{\sigma_p}(x)\neq x\).
Finally, there exists $c_*>2$ such that, for any  \(\delta\in(0,c_*^{-1}\delta_\ve/2)\) and \(y\in D_{\delta_{\ve}/2}^{N}(\tilde\Phi)\), we have
\[
\tilde\phi_{\sigma_1}\circ\cdots\circ\tilde\phi_{\sigma_p}(B_\delta(y))\cap B_\delta(y)=\emptyset.
\]
 \end{prop}
 \begin{proof} 
 The last statement of the Proposition is an immediate consequence of the first part and Lemma \ref{lem:fix}.
 As for the first part, note that if for each $\sigma\in\Sigma^m_{p, *}$, $p\leq N$, the fixed points of $\phi_\sigma:=\phi_{\sigma_1}\circ \cdots\circ \phi_{\sigma_p}$ do not belong to $D^{N}(\tilde\Phi)$, then the Proposition holds with $\delta_\ve$ small enough. Thus, it suffices to prove the latter fact.\\
 By Remark \ref{rem:nice_nbd} it follows that there exists $\ve_0>0$ such that for each $\ve\in (0,\ve_0)$, and $\ve$-perturbation $\Tilde \Phi$ of $\Phi$, for all $\sigma\in \Sigma^m_{p,*}(\Phi)$, $p\leq N$, the inverse map of $\phi_\sigma$ is well defined in $D_\sigma(\tilde \Phi)$, additionally by Lemma~\ref{lem:good_diff}, \(\Sigma^m_{p,*}(\Phi)=\sigma^m_{p,*}(\tilde\Phi)\). From now on we assume $\ve\leq \ve_0$.

 We can then apply Proposition \ref{prop:fix-sep} to obtain the IFS $\Phi^0$, a $\ve/4$ perturbation of $\phi$, where the fixed points differ unless they are associated with sequences composed by the repetition of the same word. Next, we want to proceed by induction on the sequences in $\Sigma^m_{N,*}(\Phi)$.

To this end, it is necessary to have an order structure on  $\Sigma^m_{N,*}(\Phi)$.
We introduce the following order: $0\prec \sigma$ for all $\sigma\neq 0$, if $p>q$ and $\sigma\in \Sigma^m_{p,*}(\Phi)$, $\sigma'\in \Sigma^m_{q,*}(\Phi)$, then $\sigma'\preceq \sigma$. If $p=q$, then the $\sigma$ are ordered lexicographically. This is a total ordering, hence we can arrange them as sequences $\{\sigma^i\}_{i\in\bN}$ with $\sigma^j\prec \sigma^i$ iff $i>j$. Next, define $\ell(j)$ to be the length of the word $\sigma^j$, i.e. $\sigma^j\in\Sigma^{m}_{\ell(j),*}$. Recall the definition of fixed points \eqref{eq:fix-def}, also it is convenient to set $\sigma^0=0$ and $x_0(\phi)=\emptyset$. Also let $\Lambda_0=\max\{\|(D_x\phi_i)^{-1}\|_{\cC^0(V_i)}\;:\; \phi_i\in\Phi^0\}$.\\
The idea is to define a sequence of perturbations $\Phi^k$, $\|\Phi^{k+1}-\Phi^{k}\|_{\cC^2}=\ve_k\leq \ve 2^{-k-1}$,  such that, for all $j\leq k$,
\begin{equation}\label{eq:Dk-toprove}
x_{\sigma^j}(\Phi^{k})\not \in D^{N}(\Phi^{k}).
\end{equation}
Note that the above implies that there exists $\Lambda>2$ such that  
\[
\Lambda\geq \max\{\|(D_x\phi_i)^{-1}\|_{\cC^0(U_i)}\;:\; \phi_i\in\Phi^k\} 
\]
for all $k\in\bN$. In particular, the above implies that 
\begin{equation}\label{eq:sing-bouge}
D^{N}(\Phi^{k+1})\subset D^{N}_{2\Lambda^{N}\ve_k}(\Phi^k).
\end{equation}
Using the notation introduced just before Proposition \ref{prop:fix-sep}, let
 \[
\bar \delta_k=\min\left\{1,\frac1{2c_*}\inf_{\substack{x,y\in A_{N}(\Phi^k)\\ x\neq y}}\|x-y\|\right\}>0.
\]
We proceed by induction on $\sigma$. 
For $\sigma^0$ the statement in the induction is trivially true. We assume it is true for $\sigma^k$ and we prove the statement for $\sigma^{k+1}$. Let
\[
\delta_k=\min\left \{\min\{d_0(x_{\sigma^j}(\Phi^k), D^{N}(\Phi^k))\;:\; j\leq k\}, \bar \delta_k/4\right\}.
\]
We will consider $\ve_k$-perturbations of $\Phi^k$ with $\ve_k\leq \Lambda^{-N}\delta_k/4$.
For $\sigma\in\Sigma^m_p$, we will use the notation $\phi_{\sigma}=\phi_{\sigma_1}\circ \cdots\circ\phi_{\sigma_p}$.
If $x_{\sigma^{k+1}}(\Phi^k)\not \in D^{N}(\Phi^k)$, then we set $\Phi^{k+1}=\Phi^k$ and the induction step is satisfied. 
Otherwise, as before, for some \(a>2\), and any $\delta\in (0,1/ \sqrt a), v\in V$, and $\bar x\in\bR^d$ let us define \(h_{\bar x,\delta,v}:\bR^d\to\bR^d\) as
\begin{equation}\label{eq:hdef}
			h_{\bar x,\delta,v}(x)=\begin{cases}
			x\quad&\forall x\not\in B_\delta(\bar x)\\
			x+\delta^3g(1-\delta^{-1}\|x-\bar x\|) v   &\textrm{ otherwise},
			\end{cases}
\end{equation}
where $g\in\cC^\infty(\bR,\bR_+)$ is a monotone function such that $g(y)=0$ for all $y\leq 0$, $g(y)=1$ for all $y\geq 1/2$,  and $\|g'\|_\infty< a$. \\
Let $p=\ell(\sigma^{k+1})$.
For each $\delta>0$ and $v\in\bR^d$, $\|v\|\leq 1$, we consider the perturbations
$\Phi_{\delta, v}=\{\phi_{i,\delta, v}\}$ defined by
\[
\phi_{i,\delta,v}(x)=\begin{cases}
\phi_{i}(x) &\textrm{ if } i\neq \sigma^{k+1}_p\\
\phi_{\sigma^{k+1}_p}\circ h_{x_{\sigma^{k+1}}(\Phi^k),\delta,v}(x) &\textrm{ if } i= \sigma^{k+1}_p
\end{cases}
\]
where $\Phi^k=\{\phi_i\}$. Note that $\|\phi_i-\phi_i\circ h_{\bar x,\delta,v}\|_{\cC^2}\leq C_g\delta$, and $\|\phi_i\circ h_{\bar x,\delta,v}\|_{\cC^3}\leq C_g$ for some constant $C_g>1$. Thus, these are $\ve_k$ perturbations provided $\delta\leq C_g^{-1}\ve_k\leq C_g^{-1}\Lambda^{-N}\delta_k/4$.
\begin{lem} \label{lem:fix-bouge}
There exist $ C_*>0$ and $\delta_*\in(0, \min\{C_g^{-1}\ve_k, \Lambda^{-2N}\})$ such that, for all $\delta\leq \delta_*$, $\|v\|\leq 1$ and  each $j\leq k+1$,
\[
\frac 32\delta^2\|v\|\Lambda^{-N}\leq \|x_{\sigma^{j}}(\Phi^k)-x_{\sigma^j}(\Phi_{\delta, v})\|\leq \frac{\delta}2.
\]
Moreover, $\partial_vx_{\sigma^j}(\Phi_{\delta, v})$ is invertible and
\[
\|\left(\partial_vx_{\sigma^j}(\Phi_{\delta, v})\right)^{-1}\|\leq C_*\delta^{-3}\Lambda^{N}.
\]
\end{lem}
\begin{proof}
Let $q=\ell(\sigma^j)$. If $\sigma^j_s\neq  \sigma^{k+1}_p$, for all $s\leq q$, then $x_{\sigma^{j}}(\Phi^k)=x_{\sigma^j}(\Phi_{\delta, v})$. Otherwise, let $\bar s$ be the largest such that, for all \(s\leq \bar s\),  $\sigma^j_s\neq  \sigma^{k+1}_p$  then
\[
\begin{split}
&\phi_{\sigma^j_{\bar s},\delta,v}\circ \cdots \circ\phi_{\sigma^j_q,\delta,v}\left(x_{\sigma^j}(\Phi_{\delta, v})\right)=\\
&=\phi_{\sigma^j_{\bar s},\delta,v}\circ \cdots \circ\phi_{\sigma^j_q,\delta,v}\circ \phi_{\sigma^j_1,\delta,v}\circ \cdots\circ \phi_{\sigma^j_{\bar s-1},\delta,v}\left( \phi_{\sigma^j_{\bar s},\delta,v}\circ \cdots \circ\phi_{\sigma^j_q,\delta,v}\left(x_{\sigma^j}(\Phi_{\delta, v})\right)\right).
\end{split}
\]
Note that, $(\sigma^j_{\bar s},\dots,\sigma^j_{q},\sigma^j_1,\ldots, \sigma^j_{\bar s-1},)=\sigma^{j_1}$ for some $j_1<k+1$. Moreover, $\sigma^{j}_q=\sigma^{k+1}_p$ and 
\[
x_{\sigma^{j_1}}(\Phi_{\delta,v})=\phi_{\sigma^j_{\bar s},\delta,v}\circ \cdots \circ\phi_{\sigma^j_q,\delta,v}\left(x_{\sigma^j}(\Phi_{\delta, v})\right).
\]
By hypothesis 
\[
\|x_{\sigma^{j}}(\Phi_{\delta,v})-x_{\sigma^{j}}(\Phi^k)\|\leq\Lambda^p\|x_{\sigma^{j_1}}(\Phi_{\delta,v})-x_{\sigma^{j_1}}(\Phi^k)\|.
\]
We can thus consider only the case in which $\sigma^{j}_q=\sigma^{k+1}_p$.  Let $1\leq \bar s<q$ be the largest integer, if it exists, such that $\sigma^j_s=  \sigma^{k+1}_p$ for all \(s\leq \bar s\) then for $y\in B_\delta(x_{\sigma^{k+1}}(\Phi^k))$ we have
\[
x_{\sigma^{j}}(\Phi_{\delta,v})=\phi_{\sigma^j_{1},\delta,v}\circ \cdots \circ\phi_{\sigma^j_{\bar s},\delta,v}
\circ \phi_{\sigma^j_{\bar s+1}}\circ \cdots \circ\phi_{\sigma^j_{q}}\left(h_{x_{\sigma^{k+1}}(\Phi^k),\delta,v}(y)\right).
\]
Since, by construction, $h_{x_{\sigma^{k+1}}} (B_\delta(x_{\sigma^{k+1}}(\Phi^k)))\subset B_\delta(x_{\sigma^{k+1}}(\Phi^k))$ we have that applying $\phi_{\sigma^j_{\bar s},\delta,v}$ differs from applying $\phi_{\sigma^j_{\bar s}}$ only if
\[
\phi_{\sigma^j_{\bar s+1}}\circ \cdots \circ\phi_{\sigma^j_{q}}\left(B_\delta(x_{\sigma^{k+1}}(\Phi^k))\right)
\bigcap B_\delta(x_{\sigma^{k+1}}(\Phi^k))\neq \emptyset.
\]
But then Lemma \ref{lem:fix} implies that 
\[
\|x_{(\sigma^j_{s+1}, \dots, \sigma^j_q)}(\Phi^k)-x_{\sigma^{k+1}}(\Phi^k)\|\leq c_*\delta.
\]
Then our choice of $\delta$ and the induction hypothesis implies that $x_{\sigma^{k+1}}(\Phi^k)=x_{(\sigma^j_{s+1}, \dots, \sigma^j_q)}(\Phi^k)\not \in D^{N}(\Phi^k)$ contrary to our current assumption. It follows that, provided $x_{\sigma^{j}}(\Phi_{\delta,v})\in B_\delta(x_{\sigma^{k+1}}(\Phi^k))$,
\[
x_{\sigma^{j}}(\Phi_{\delta,v})=\phi_{\sigma^j}\left(h_{x_{\sigma^{k+1}}(\Phi^k),\delta,v}(x_{\sigma^{j}}(\Phi_{\delta,v}))\right).
\]
To simplify notation let $z(\delta,v)=x_{\sigma^{j}}(\Phi_{\delta,v})$, $h_{\delta,v}=h_{x_{\sigma^{k+1}}(\Phi^k),\delta,v}$ and $\phi_{\delta,v}=\phi_{\sigma^j}\circ h_{\delta,v}$. We can study $z(\delta,v)$ applying the implicit function theory which yields 
\[
\frac{d}{d\delta} z(\delta,v)=(\Id-D_{z(\delta,v)}\phi_{\delta,v})^{-1}D_{h_{\delta,v}(z(\delta,v))}\phi_{\delta,v}\partial_\delta h_{\delta,v}(z(\delta,v)).
\]
If $z(\delta,v)\in B_{\delta/2}(x_{\sigma^{k+1}}(\Phi^k))$, then $h_{\delta,v}(z(\delta,v))=z(\delta,v)+\delta^3 v$, $D_{z(\delta,v)} h_{\delta,v}=\Id$, and $\partial_\delta h_{\delta,v} (z(\delta,v))=3\delta^2v$. Thus, setting $A:=D_{z(\delta,v)+\delta^3v}\phi_{\delta,v}$ we obtain
\[
\frac{d}{d\delta} z(\delta,v)=3\delta^2(\Id-A)^{-1}Av.
\]
Since the maximal eigenvalue of $(\Id-A)^{-1}A$ is bounded by $(1-\lambda)^{-1}\lambda$, there exists a $\delta_*$ such that, for all $\delta\leq \delta_*$, $z(\delta,v)\in B_{\delta/2}(x_{\sigma^{k+1}}(\Phi^k))$. Moreover, 
\[
\|(\Id-A)^{-1}Av\|\geq \Lambda^{-N}/2\|v\|
\]
thus $z(\delta,v)\not \in B_{\frac 32\delta^2\|v\|\Lambda^{-N}}(x_{\sigma^{k+1}}(\Phi^k))$.
Finally, for each $\delta\leq \delta_0$, we have
\[
\partial_vz(\delta,v)=\delta^3(\Id-A)^{-1}A
\]
from which the last statement of the Lemma follows.
\end{proof}
By Lemma \ref{lem:fix-bouge}, equation \eqref{eq:sing-bouge}, and our choice of $\ve_{k}$, $x_{\sigma^j}(\Phi_{\delta,v})\not \in D^{N}_{3\delta_k/4}(\Phi^{k})$  and $D^{N}(\Phi^{k+1})\subset D^{N}_{\delta_k/4}(\Phi^{k})$ for all $j\leq k$. Thus, $x_{\sigma^j}(\Phi_{\delta,v})\not \in D^{N}(\Phi^{k+1})$ for all $j\leq k$. 
We are left with $x_{\sigma^{k+1}}(\Phi^{k+1})$, recalling that $x_{\sigma^{k+1}}(\Phi^k) \in D^{N}(\Phi^k)$.
Let $\omega\in \Sigma^m_{q,i}$, $q\leq N$ and $M \in M_{\omega,i}^{N}(\Phi^k)$ such that $x_{\sigma^{k+1}}(\Phi^k)\in M$, then
\[
\phi_{\omega_1}\circ\cdots\circ\phi_{\omega_q}(x_{\sigma^{k+1}}(\Phi^k))\in \overline{\psi_i(D_i)}.
\]
Suppose first that $\omega_q\neq \sigma^{k+1}_p$.
It follows that for $y\in B_{\delta/2}(x_{\sigma^{k+1}}(\Phi^k))$
\[
\phi_{\omega_1, \delta,v}\circ\cdots\circ\phi_{\omega_q, \delta,v}(y)\neq \phi_{\omega_1}\circ\cdots\circ\phi_{\omega_q}(y)
\]
only if for some $s< q$, $\omega_s=\sigma^{k+1}_p$ and
\begin{equation}\label{eq:nogo}
\phi_{\omega_{s+1}}\circ\cdots\circ\phi_{\omega_q}(B_{\delta}(x_{\sigma^{k+1}}(\Phi^k)))\cap B_{\delta}(x_{\sigma^{k+1}}(\Phi^k))\neq \emptyset,
\end{equation}
but this is ruled out by our choice of $\delta_k$ and Proposition \ref{prop:fix-sep}. The above discussion shows that $M\cap B_{\delta/2}(x_{\sigma^{k+1}}(\Phi^k))$ is an element of $M_{\omega, i}^{N}(\Phi_{\delta,v})$ as well. Hence, it suffices to ensure that $x_{\sigma^{k+1}}(\Phi_{\delta,v})\not \in M$. Since Lemma \ref{lem:fix-bouge} shows that varying $v$ the fixed point visits an open ball, and since $M$ has zero measure, it follows that there exists an open set of $v$ which yields the wanted property.\\
 It remains to analyse the case $\omega_q= \sigma^{k+1}_p$. In this case, for  $y\in B_{\delta/4}(x_{\sigma^{k+1}}(\Phi^k))$ we have
\[
\phi_{\omega_1, \delta,v}\circ\cdots\circ\phi_{\omega_q, \delta,v}(y)=\phi_{\omega_1, \delta,v}\circ\cdots\circ\phi_{\omega_q}(y+\delta^3 v)\neq \phi_{\omega_1}\circ\cdots\circ\phi_{\omega_q}(y+\delta^3 v)
\]
only if equation \eqref{eq:nogo} is satisfied, which, by our choice of $\delta_k$ is possible only if $x_\omega(\Phi^k)=x_\sigma(\Phi^k)$. But then Proposition \ref{prop:fix-sep} implies that there exist $\sigma^r\preceq \sigma^{k+1}$ such that $\phi_\omega=\phi_{\sigma^r}^{m_1}$ and $\phi_\sigma=\phi_{\sigma^r}^{m_2}$, which would mean
\[
\overline{\psi_i(D_i)}\ni \phi_{\sigma^r}^{m_1}(x_{\sigma^{k+1}}(\Phi^k))=x_{\sigma^{k+1}}(\Phi^k).
\]
Again, Lemma \ref{lem:fix-bouge} allows to find an open set of $v$ for which $x_{\sigma^{k+1}}(\Phi^k)\not\in \overline{\psi_i(D_i)}$.
The last possibility is that 
\[
\phi_{\omega_1, \delta,v}\circ\cdots\circ\phi_{\omega_q, \delta,v}(y)=\phi_{\omega_1, \delta,v}\circ\cdots\circ\phi_{\omega_q}(y+\delta^3 v).
\]
This implies that the perturbed manifold $M$ is displaced by, at most, $2\Lambda^{N}\delta^3\|v\|$ while Lemma \ref{lem:fix-bouge} implies that the fixed point moves by at least $\frac 32\delta^2\|v\|\Lambda^{-N}\geq 2\Lambda^{N}\delta^3\|v\|$. Hence, again we have an open set of $v$, which produce perturbations with the wanted property. As a last observation, note that if there are other manifolds $M\in D^{N}(\Phi^k)$ such that $x_{\sigma^{k+1}}(\Phi^k)\in M$, then we can repeat the same argument and we have just a smaller open set of $v$ that does the job.
This concludes the overall induction and hence the proof of Proposition \ref{prop:bad_points}.
\end{proof}

\subsection{Perturbations with low complexity and the proof of Proposition~\ref{prop:psc-dense}}\label{sec:4.3}\ \\
Thanks to Proposition~\ref{prop:bad_points} we can finally construct the wanted perturbation $\tilde f$.\\
	 Let \(f\) be a piecewise smooth contraction with the maximal partition \(\bsP(f)=\{P_1,P_2,\ldots,P_m\}\). Let \(l_{0}\in\bN\) be the number of manifolds in \(\partial\bsP(f)\). Define \(l_1=\max\{c_d l_{0},d\}\).\footnote{ Recall that, by construction $c_d$ is the maximal number of manifolds in $M^N_{\sigma,i}$, $N\in\bN$ and $\sigma\in\Sigma^m_{0,i}$, that can contain a point in $\psi_i(D_i^+)$.} 
	 
Given two manifolds defined by maps $\psi_1,\psi_2$, we write $\psi_1\pitchfork\psi_2$ if the manifolds are transversal, see \ref{def:transversal} for the definition of transversality. On the contrary, if the two manifolds have an open (in the relative topology) intersection, we call them {\em compatible} and write $\psi_1\curlywedge \psi_2$. If two manifolds are not compatible, we write $\psi_1\notcurlywedge \psi_2$.

 \begin{prop}\label{lem:good_partn}
     Let \(f:X\to X\) be a piecewise smooth contraction with maximal partition \(\bsP(f)\). Then, for any \(N\in\bN\) and \(\ve>0\) small enough, there exists a piecewise smooth contraction \(\tilde{f}\), with \(d_2(f,\tilde{f})<\ve\) such that no more than $2^{dm^{d-1} l_1}$ partition elements of $\bsP(\tilde f^N)$ can meet at one point.
 \end{prop}
 \begin{proof}
Before starting the proof, we need to introduce some language.\\
Consider an IFS \(\Phi_{\! f}=\{\phi_1,\phi_2,\ldots,\phi_m\}\) associated to $f$ with contraction coefficient \(\lambda\). Let \(\ve_0\in(0,\ve/4)\) be small enough, then by Proposition~\ref{prop:bad_points} there exist \(\delta_*\leq \delta_\ve\in (0,\ve_1)\)  and an \(\ve_0-\)perturbation\footnote{ Here, and in the following by {\em perturbation} we mean a function which is $\cC^2$ close and with a uniformly bounded $\cC^3$ norm.} of \( \Phi_{\! f}\) (which, abusing the notation, we still call \(\Phi_{\! f}\)), such that, for every \(p\leq N\), \(\sigma=(\sigma_1,\ldots,\sigma_p)\in\Sigma^{m}_{p,*}\) and \(\xi\in D^{N}_{\delta_*}(\Phi_{\! f})\cap V_p\),
 \begin{equation}\label{eq:disjt}
	 \phi_{\sigma_1}\circ\phi_{\sigma_{2}}\circ\cdots\circ\phi_{\sigma_p} (B_{\delta_*}(\xi))\cap B_{\delta_*}(\xi)=\emptyset.
 \end{equation} 
 Note that there exists $\epsilon_0\leq \ve_0$ such that \eqref{eq:disjt} persists for  $\epsilon_0$-perturbations of $\Phi$.\\
By compactness, for \(\delta\in(0,\min\{\delta_{*}/2,\delta_N\})\), where \(\delta_N>0\) is such that it satisfies the condition of Remark~\ref{rem:nice_nbd} for each $\epsilon_0$-perturbation, there exists a finite open cover \(\{B_{\delta/2}(z_i)\}_{i=1}^t\) of \(\overline{D_{\delta}^N(\Phi_{\! f})}\) (which, by definition, contains \(\partial  \bsP(f^N) \)) such that, for each $i$, $z_i\in P_j$ and $B_{\delta}(z_i)\subset V_j$, for some $j$.\footnote{ See the discussion at the beginning of Section \ref{sec:manifold} for the definition of $V_j$.}  \\ 
Let $Y=\overline Y\subset\bR^d$ be compact, such that $\overline{\phi (Y)}\subset Y$ for all $\phi\in\Phi$. 

By convention, we set $\cZ^1_0(\Phi)=\{Y\}$ and $\cZ^1_1(\Phi):=\{\psi_{\omega}\;;\;\omega\in\{1,\dots,l_{*}\}\}$ be the collection of the manifolds $\psi_i(A)$ for $A\in \cW^N_i$, as defined in \eqref{eq:many-manifolds}.
Also, we call $\cZ^1_k(\Phi):=\{\psi^k_{\omega}\;:\;\omega\in\{1,\dots,l_{*}\}^k\}$ the manifolds consisting of the intersection of the $k$ manifolds $\{\psi_{\omega_i}\;:\;i\in\{1,\dots,k\}; i\neq j\implies\psi_{\omega_i}\notcurlywedge\psi_{\omega_j}\}$.\footnote{ These are indeed manifolds, see Definition \ref{def:psc}. To simplify notations, we use  $\psi_\omega$ both for the manifold and for the map that defines it.} 
 By construction, $\cZ^1_k(\Phi)=\emptyset$ for $k>l_0$. In addition, the maximal dimension of the manifolds in $\cZ^1_k(\Phi)$, for $k>1$, is $d-2$ (since the boundary manifolds are pairwise transversal, see Definition \ref{def:psc}). Note that $\cZ^1_1(\Phi_{\! f})$ is a collection that covers the boundary manifolds; for simplicity, we will call the elements of $\cZ^1_1(\Phi_{\! f})$, from now on, original boundary manifolds.
For each $s\in\bN$ and $k_1,\dots, k_n\in\bN\cup\{0\}$, let $\cZ^s_0=\{Y\}, U_0=Y$ and\footnote{ To alleviate notation, from now on we will write $\phi_{\sigma_i}^{-1}$ to mean the inverse of $\phi_{\sigma_i}|_{U_{\sigma_i}}$, while the domain of $\phi_{\sigma_i}^{-1}\circ \psi$ consists of the points where the composition is well defined.}
\[
\begin{split}
&\cZ_{k}^{s+1}(\Phi)=\Big\{\phi_{\sigma_1}^{-1}\circ\psi_{\omega_1}\cap \cdots\cap \phi_{\sigma_n}^{-1}\circ\psi_{\omega_n}\;:\; n\in\bN, k_1,\dots, k_n\in\bN\cup\{0\}, \sum_{i=1}^nk_i=k, \\
&\phantom{\cZ_{k}^{s+1}(\Phi)=\Big\{\ \ }
\psi_{\omega_i}\in\cZ_{k_i}^s(\Phi), \sigma_i\in\Sigma^m_{1,*}(\Phi); i\neq j,\sigma_i=\sigma_j \implies \psi_{\omega_i}\notcurlywedge\psi_{\omega_j}\Big\}\\
&\cZ_{*}^{s}(\Phi)=\bigcup_{k\in\bN\cup\{0\}}\cZ_{k}^{s}(\Phi).
\end{split}
\]
Note that $\cZ^s_*$ contains the admissible pre-images of the boundary manifolds under composition of at most $s$ maps in $\Phi$ and all their intersections. In particular, the sets in $\cZ^s_1$ cover $\partial \cP(f^s)$, and we will call them boundary manifolds. In addition, if a set belongs to $\cZ^s_k(\Phi)$, then, by definition, it is determined by the intersection of the pre-images of $k$ original boundary manifolds. Also, remark that $\cZ^{s+1}_k\supset \cZ^s_k$ since $\phi_0=id$ (see Remark \ref{rem:boundary0} and Definition \eqref{eq:sigmazero}). Next, let $N_s$ be the maximal number of manifolds that can intersect in $\cZ^s_*$, that is, $\cZ^s_k=\emptyset$ for $k>N_s$. We have seen that $N_1\leq l_1$. Moreover, each original boundary manifold can have at most $m^s$ different pre-images obtained by the compositions of $s$ maps. This implies that, at each point, we have at our disposal at most $\sum_{s'=0}^{s-1} m^{s'}l_1$ different manifolds from $\cup_i M^N_{0,i}$ to intersect. If $m=1$ then $N_s\leq sl_1$; if $m\geq 2$, then $N_s\leq\frac{m^{s}-1}{m-1} l_1< 2m^{s-1}l_1$. Accordingly,  $N_s\leq s m^{s-1} l_1$.\footnote{ Remark that, by definition, the pre-images are taken via invertible maps, hence the manifolds cannot self-intersect.}

Our goal is to produce a sequence of perturbations $\Phi^s$ of $\Phi_{\! f}=:\Phi^0$ such that $\Phi^{s}$ is a $2^{-s}\ve$ perturbation of $\Phi^{s-1}$ with the following property:
\begin{enumerate}[($\star$)]
\item {\em the set $\cZ^{s}_k(\Phi^s)$ consists of manifolds of dimension strictly smaller than $d-j$ for all $k> jm^{j-1}l_1$. While $\cZ^{s}_1(\Phi^s)$ consists of $d-1$ dimensional manifolds. This property persists for small perturbations of $\Phi^s$.}
\end{enumerate}
Note that the above implies that $\cZ^{s}_k(\Phi^{s})=\emptyset$, for  each $s\in\bN$ and $k> dm^{d-1}l_1$. Accordingly, at most $dm^{d-1} l_1$ pre-images of the original boundary manifolds under composition of at most $s$ elements of $\Phi^s$ can have nonempty intersections. In turn, defining $f_s(x)=\phi_i(x)$ for $x\in P_i$ and $\phi_i\in \Phi^s$, we obtain a perturbation of $f$ smaller than $\sum_{j=1}^{s}2^{-j}\ve\leq \ve$, such that $\cP(f_s^s)$ has at most $2^{dm^{d-1} l_1}$ elements meeting at point.\\
 Indeed, suppose $p$ elements of  $\cP(f_s^s)$ meet at a point $x$. The boundaries of such elements in a neighbourhood small enough of $x$ consist of co-dimension one manifolds belonging to $\cZ^s_1(\Phi^s)$, and they have to intersect at $x$. Suppose the total number of such boundary manifolds is $q$, then $x$ must belongs to a manifold in $\cZ^s_q(\Phi^s)$, hence it must be $q\leq dm^{d-1} l_1$. Note that we can uniquely define a partition element by specifying on which side it lies with respect to all its boundary manifolds. Since there are at most $2^q$ possibilities, it must be $ p\leq 2^q$. It follows that $p\leq 2^{dm^{d-1} l_1}$. The Lemma would then follow by choosing $s=N$.

It remains to prove property $(\star)$, we will proceed by induction. If $s=1$, and $k\in\{1,\dots,l_1\}$, then the manifolds in $\cZ^1_k$ are indeed of co-dimension at least one, with the manifolds in $\cZ^1_1$ of co-dimension one, while if $k >l_1$ then $\cZ^1_k=\emptyset$, so $\Phi_{\! f}=\Phi^0$ satisfies our hypothesis. Let us assume that the hypothesis is verified for some $s$ and prove it for $s+1$.

Let $\epsilon_s\leq \min\{2^{-s-1}\ve,\epsilon_0\}$ be such that all the $\epsilon_s$-perturbations of $\Phi^s$ still satisfy $(\star)$. This implies that, provided $\Phi^{s+1}$ is a $\epsilon_s$ perturbation of $\Phi^s$, $\cZ_*^{s'}(\Phi^{s+1})$ has the wanted property for all $s'\leq s$.

Accordingly, we must analyse only sets of the type $\phi_{\sigma_1}^{-1}\circ\psi_{1}\cap \cdots\cap\phi_{\sigma_n}^{-1}\circ\psi_n$, where $\psi_i\in\cZ_{k_i}^s(\Phi^s)$ and $\phi_{\sigma_i}\in\Phi^{s}\cup\{id\}$, $\sigma_i\in\{0,\dots,m\}$.\footnote{ By an innocuous abuse of notation here we use $\psi_i$ to refer to generic elements.} By definition, such sets are elements of $\cZ^{s+1}_k(\Phi^s)$, with $k=\sum_{i=1}^n k_i$.
Note that the $\psi_i\in \cZ^s_0(\Phi^s)$ do not contribute to the intersection, we can thus assume w.l.o.g. that $k_i>0$.  Note that if $n=1$, then the manifolds belong to $\cup_{i=0}^m\phi_i^{-1}(\cZ_*^s(\Phi^s))\subset  \cZ_*^{s+1}(\Phi^{s})$ which have automatically the wanted property, and so has any $\epsilon_s$-perturbation. We consider thus only the case $n\geq 2$.
In addition, if $\phi_{\sigma_i}=\phi_{\sigma_j}$, $i\neq j$, then $\phi_{\sigma_i}^{-1}\circ \psi_i\cap\phi_{\sigma_j}^{-1}\circ \psi_j=\phi_{\sigma_i}^{-1}\circ (\psi_j\cap\psi_j)$ and, since by definition $\psi_i\notcurlywedge\psi_j$, $\psi_j\cap\psi_j\in \cZ^s_{k_i+k_j}(\Phi^s)$. Hence, we can substitute to the intersection of the manifolds $\phi_{\sigma_i}^{-1}\circ \psi_i\cap\phi_{\sigma_j}^{-1}\circ \psi_j$ the manifold $\phi_{\sigma_i}^{-1}\circ (\psi_j\cap\psi_j)$. We can thus assume w.l.o.g. that $i\neq j$ implies $\phi_{\sigma_i}\neq \phi_{\sigma_j}$.\\
Define the map $F:\bR^d\to\bR^{nd}$ by $F(x):=(\phi_{\sigma_1}(x),\dots, \phi_{\sigma_n}(x))$ and the stratified subvariety $C=\{(\psi_{1}(x_0), \dots\psi_{n}(x_n))\;:\; x_i\in \overline D_{i}\}$, where $D_i\subset \bR^{d_i}$ is the domain of the map $\psi_i$. By Lemma \ref{lem:intersection}, for a constant $c$ to be chosen later, there exists a $\frac 12 c\epsilon_s$-perturbation $\hat F=(\hat \phi_{\sigma_1}(x),\dots, \hat\phi_{\sigma_n}(x))$ of $F$, transversal to $C$.
 If $\phi_{\sigma_i}\neq id$ for all $i$, then we set $\tilde F=\hat F$. If, for some $i$, $\phi_{\sigma_i}=id$, then $\hat \phi_{\sigma_i}$ is a small perturbation of identity, and hence it is invertible with $\cC^3$ inverse.\footnote{ Indeed if $\|h-id\|_{\cC^1}=\alpha<1$ then $h$ is a diffeo. In fact, if $h(x)=h(y)$,
\[
0=\int_0^1\frac{d}{dt}h(ty +(1-t)x) dt=y-x+\int (Dh-\Id)(x-y)dt
\]
which implies $\|x-y\|\leq \alpha \|x-y\|$, that is $x=y$. Thus, $h$ is globally invertible, and the claim follows by the inverse function theorem. } By possibly relabelling, we can assume $i=1$. Then we define
\[
 \tilde F(x)=(\tilde \phi_{\sigma_1}, \dots, \tilde \phi_{\sigma_n})=(x,\hat\phi_{\sigma_2}\circ\hat \phi_{\sigma_1}^{-1},\dots,\hat\phi_{\sigma_n}\circ\hat \phi_{\sigma_1}^{-1}).
 \]
 $\tilde F$ is still transversal to $C$ and, if $c$ is small enough, by Lemmata \ref{lem:good_pert} and \ref{lem:good_pert2}, it is a $c\epsilon_s$ perturbation of $F$. Let $d_i=d-j_i$ be the dimension of the manifold $\psi_{i}$, then $k_i\leq j_im^{j_i-1}l_1$. Lemma \ref{lem:intersection}  implies that the sets $\bigcap_{i=1}^n\phi_{\sigma_i}^{-1}\circ\psi_i$ are manifolds with dimension\footnote{ Actually, they are stratified sub-varieties, but we can restrict them to manifolds without loss of generality.} 
 \[
 \sum_{i=1}^n (d-j_i)-(n-1)d=d-\sum_{i=1}^nj_i=:d-\bar j\leq d-\max\{j_i\}-1.
 \]
 Note that
 \[
 k:=\sum_{i=1}^nk_i\leq \sum_{i=1}^n j_im^{j_i-1}l_1\leq \sum_{i=1}^n j_im^{\bar j-1}l_1=\bar j m^{\bar j-1}l_1.
 \]
Accordingly, if $k>j m^{j-1}l_1$, then $\bar j\geq j+1$ and the manifold has a dimension strictly smaller than $d-j$,
 as required.\footnote{ Of course, if $k>dm^{d-1}l_1$, then the intersection is empty.}\\
 We would then like to define a perturbed IFS $\tilde\Phi^{s}$ as
 \[
\tilde \phi_{k}=\begin{cases} 
\tilde  \phi_{\sigma_i}&\textrm{ if }k=\sigma_i\\
\phi_k&\textrm{ otherwise}.
\end{cases}
\]
Unfortunately, this may perturb the new manifolds $\psi_i$ as well, since they are now defined via pre-images of $\tilde{\Phi}^{s}$. This is the last problem we need to take care of. To this end, we make the perturbation only locally starting from the ball $B_{\delta}(z_1)$. Once we check that the perturbation is as required in this ball, we will consider the other balls, making the new perturbations small enough not to upset the property obtained in $B_{\delta/2}(z_1)$.

Let \(g\in \cC^\infty(\bR^d,[0,1])\)  be such that
	     \[
	     g(z)=
		 \begin{cases}1\quad\quad z\in B_{\delta/2}(z_1)\\
		 0 \quad\quad z\notin B_{3\delta/4}(z_1)
	     \end{cases}
	     \]
     and $\|g\|_{\cC^r}\leq C\delta^{-r}$ for $r\in\{0,1,2,3\}$ for some \(C>0\). Define, for each $i\in\{1,\dots,m\}$, 
 \[
	 \phi_{i,1}(x)=\phi_{i}(x)+g(x)(\widetilde \phi_i(x)-\phi_i(x)).
\]
Provided we choose $c$ small enough
\[
	 \norm{ \phi_{i,1}-\phi_{i}}_{\cC^2}=\norm{g}_{\cC^2}\norm{\widetilde \phi_i-\phi_i}_{\cC^2}<C\delta^{-2}C_\#c\epsilon_s <\epsilon_s/4.
\]
We define then the perturbation $\Phi^{s,1}=\{\phi_{i,1}\}$. Note that the  $\Phi^{s,1}$ equals $\Phi^{s}$ outside the ball $B_\delta(z_1)$ and agrees with $\tilde\Phi^s$ inside the ball $B_{\delta/2}(z_1)$.

Recall that we perturbed the system in order to control the intersection of the manifolds $\phi_{\sigma_1}^{-1}\circ\psi_1\cap \cdots\cap\phi_{\sigma_n}^{-1}\circ\psi_n$.
By constructions each $\psi_i$ is the intersection of manifolds $\phi_{\sigma^{i,j}_{s}}^{-1}\circ\cdots\phi_{\sigma^{i,j}_1}^{-1}\circ  \bar \psi_{i,j}$ with $\bar \psi_{i,j}\in\cZ_1^1(\Phi^s)$ and $\sigma^{i,j}\in\{0,\dots,m\}^s$. We are thus interested in $A:=\phi_{\sigma_{i},1}^{-1}\circ \phi_{\sigma^{i,j}_{s},1}^{-1}\circ\cdots\circ\phi_{\sigma^{i,j}_1,1}^{-1}\circ  \bar \psi_{i,j}\cap B_{\delta}(z_1) $ which are the perturbation of $\phi_{\sigma_i}^{-1}\circ\psi_i$. Let 
\[
h_i(x)=\begin{cases}
x&\textrm{ if } x\not\in B_\delta(z_1)\\
\phi_{\sigma^{i}}^{-1}\circ \phi_{\sigma^{i},1}&\textrm{ otherwise}.
\end{cases}
\] 
By choosing $\epsilon_s$ small enough we have $\|h_i- id\|_{\cC^2}\leq \delta/4$, in particular $h_i$ is a diffeomorphism. It follows that $h_i(B_\delta(z_1))\subset B_\delta(z_1)$. 
Let $x\in A$, then Proposition \ref{prop:bad_points} implies 
\[
\begin{split}
\phi_{\sigma^{i,j}_{1},1}\circ\cdots\phi_{\sigma^{i,j}_s,1}\circ \phi_{\sigma_i,1}(x)&=\phi_{\sigma^{i,j}_{1},1}\circ\cdots\phi_{\sigma^{i,j}_s,1}\circ \phi_{\sigma_i}\circ h_i(x)\\
&=\phi_{\sigma^{i,j}_1}\circ\cdots\circ \phi_{\sigma^{i,j}_s}\circ \phi_{\sigma_i}\circ h_i(x)\\
&=\phi_{\sigma^{i,j}_1}\circ\cdots\circ \phi_{\sigma^{i,j}_s}\circ\phi_{\sigma^{i},1}(x)
\end{split}
\]
hence the part of $\psi_i$ contained in $\phi_{\sigma^{i},1}(B_\delta(z_1))$ is unchanged. This implies that, inside the ball $B_{\delta/2}(z_1)$ the IFS $\Phi^{s,1}$ has the wanted property for the manifold $\phi_{\sigma_{1},1}^{-1}\circ\psi_1\cap \cdots\cap\phi_{\sigma_n,1}^{-1}\circ\psi_n$. Moreover, by the openness of the transversality property, there exists $\epsilon_{s,1}\leq \epsilon_s/4$ such that the wanted property persists in $B_{\delta/2}(z_1)$ for each $\epsilon_{s,1}$ perturbation. We can now consider all the other pre-images and do the same procedure with $\epsilon_{s,j}\leq 4^{-j-1}\epsilon_s$ for the $j$-th intersection manifold. In this way, we can construct an IFS $\Phi^{s,q}=:\tilde\Phi^{s,1}$, for some $q\in \bN$ which is a $\epsilon_s/3$ perturbation of $\Phi^s$ and has the wanted property in $B_{\delta/2}(z_1)$. We can then repeat the same procedure in the ball $B_\delta(z_2)$ obtaining and IFS $\tilde \Phi^{s,2}$ which is a $ \epsilon_s/9$ perturbation of $\tilde \Phi^{s,1}$, small enough not to upset what we have achieved in $B_{\delta/2}(z_1)$. Iterating such a construction we finally obtain $\Phi^{s+1}=\tilde\Phi^{s,t}$, which has the wanted property on all the space since $\{B_{\delta/2}(z_i)\}_{i=1}^t$ is a covering of \( D^N(\Phi^s_{\! f})\), and is a $\sum_{i=1}^{t}\epsilon_s 3^{-k}\leq \epsilon_s\leq 2^{-s-1}\ve$ perturbation of $\Phi^s$. This concludes the induction argument.
\end{proof}

Finally, we can prove Proposition~\ref{prop:psc-dense}

\begin{proof}[\bf \em Proof of Proposition~\ref{prop:psc-dense}]\label{pf:psc-dense}
	 Let \(f\) be a piecewise smooth contraction with contraction coefficient \(\lambda<1\) and maximal partition \(\bsP(f)=\{P_1, P_2, \ldots, P_m\}\). If \(\lambda<1/2m\) then we are done, else the following: \\
	 Let the IFS associated to \(f\) be \(\Phi_{\! f}=\{\phi_1,\phi_2,\ldots,\phi_m\}\) and let \(l_1=\max\{c_d l_0,d\}\) where \(l_0\) is the number of boundary manifolds in \(\partial\bsP(f)\), \(d\) is the dimension of the space, and $c_d$ is the maximum number of original boundary manifolds overlaps (as defined in \eqref{eq:many-manifolds}). Let \(N\in\bN\) be the least number such that \(\lambda^N 2^{ dm^{d-1} l_1} <1/4\). By Lemma~\ref{lem:good_partn}, for \(\ve>0\) small enough,  there exists a piecewise contraction \(\tilde{f}\) such that \(d_2(f,\tilde{f})<\ve\) and no more than \(2^{dm^{d-1} l_1}\) elements of the partition \(\partial\bsP(\tilde f^N)\)  have a non-empty intersection of their closure.\\
	  Accordingly, for each $x\in X$ there is a $\delta(x)$ such that $B_{\delta(x)}(x)$ intersects at most $2^{dm^{d-1} l_1}$ elements of \(\bsP(\tilde f^N)\). Since $X$ is compact, we can extract a finite cover 
     $\{B_{\delta(x_j)/2}(x_j)\}$. Set $\delta=\frac 12\min\{\delta(x_j)\}$ and let $k\in\bN$ be such that for any partition element \(P\in\bsP(\tilde{f}^{kN})\), \(\operatorname{diam}(\tilde{f}^{kN}(P))<\delta/2\); hence \(\tilde{f}^{kN}(P)\subset B_{\delta(x_j)}(x_j)\) for some $j$, therefore it can intersect at most $2^{dm^{d-1} l_1}$ elements of \(\bsP(\tilde f^N)\).\\
     To conclude, let  \(L\) be the number of elements of \(\bsP(\tilde{f}^{kN})\). Then \(\#\bsP(\tilde{f}^{2kN})\leq L2^{dm^{d-1} l_1}\) and \(\#\bsP(\tilde{f}^{jkN})\leq L(2^{dm^{d-1} l_1})^{j}\) for \(j\in\bN\). Since \(\lambda^{kN}(2^{dm^{d-1} l_1})<1/4\), there exists \(j_*\in\bN\) such that \(L(2^{dm^{d-1} l_1})^{j_*}\lambda^{j_*kN}< 1/2\). Hence \(\tilde{f}^{j_*kN}\) is strongly contracting. 
\end{proof}

\appendix
\section{Extension Theorem}
Here we discuss the extension theorems needed in the paper. Recall the classical extension theorem for Lipschitz functions.
\begin{thm}[Kirszbraun-Valentine Theorem]\cite{Vt}\label{thm:KV}
	 Let \(f:S(\subset \bR^d)\to \bR^d\) be a Lipschitz continuous function then \(f\) can be extended to any set \(T \subset \bR^d\) to a Lipschitz continuous function with the same Lipschitz constant.
\end{thm}
The above can be easily extended to $\cC^r$ functions.
\begin{thm}[\(\cC^r\) version of Kirszbraun-Valentine Theorem]\label{thm:sKV}
	 Let \(S\subset \bR^d\) be a compact set and \(f:S(\subset \bR^d)\to \bR^d\) be a \(\cC^r\) function, for \(r\in\bN\), such that \(Lip(f)=\lambda<1\) and \(f^{-1}|_S\) be \(\cC^1\) then \(f\) can be extended to \( \bR^d\) to a \(\cC^r\) function \(f_*\) such that \(Lip(f_*)=Lip(f)\) and \(f_*^{-1}|_S=f^{-1}|_S\).
\end{thm}
\begin{proof}
	 Note that \(\norm{f|_S}_{\cC^r}\) being finite and \(S\) being compact implies that there exists an open neighbourhood \(U\) of \(S\) such that \(f\) is \(\cC^r\) in \(U\). By Inverse function theorem, \(f\) is invertible in \(U\) with \(\norm{f^{-1}|_{U}}_{\cC^1}<\infty\).  By Kirszbraun-Valentine Theorem~\ref{thm:KV}, there exists \(\tilde{f}: \bR^d\to \bR^d\) such that \(\tilde{f}|_S=f\) and \(Lip(\tilde{f})=Lip(f)=\lambda\). Then \(\tilde{f}|_{U}\circ f^{-1}|_{U}=id\) so we can define \({\tilde{f}}^{-1}|_{U}=f^{-1}\). Now, let \(\phi:\bR^d\to\bR^d\) be a \(\cC^\infty\) function compactly supported on \(\overline{U}\) and \(\int\phi=1\) and define the convolution 
	 \[
	 \tilde{f}\star \phi(x)=\int\tilde{f}(x-y)\phi(y)dy.
	 \]
	 Let, for \(\delta>0\), \(V\) be a \(\delta\)-neighbourhood of \(U\).  
	 Define a \(\cC^\infty\) function \(g:\bR^d\to \bR\) as 
	     \[
		 g(x)=
		 \begin{cases}
		 1,\quad\quad x \in	U\\
         0,\quad\quad x \notin  V
		 \end{cases}
		 \]
	 such that \(\norm{g}_{\cC^r}<c_r\). Finally, define \(f_*:\bR^d\to\bR^d\) as \(f_*(x)=g(x)f(x)+(1-g(x))\tilde{f}\star\phi(x)\). Let \(x\in U\) then
	     \[
		 f_*(x)=g(x)f(x)+(1-g(x))\tilde{f}\star\phi(x)=f(x)+0=f(x),
		 \]
		 therefore, \(f_*\) is an extension of \(f\). To check if it is Lipschitz, let \(x_1,x_2\in\bR^d\) then 
         \[
		 \begin{split}
		 &\norm{f_*(x_1)-f_*(x_2)}_{\infty}=\\
		 &=\norm{g(x_1)f(x_1)+(1-g(x_1))\tilde{f}\star\phi(x_1)-g(x_2)f(x_2)-(1-g(x_2))\tilde{f}\star\phi(x_2)}_\infty\\
		 &\leq\norm{g(x_1)(f(x_1)-\tilde{f}\star\phi(x_1))-g(x_2)(f(x_2)-\tilde{f}\star\phi(x_2))}_\infty\\
		 &\hskip12pt+\norm{\tilde{f}\star\phi(x_1)-\tilde{f}\star\phi(x_2)}_\infty\\
		 &\leq 0+\norm{\tilde{f}\star\phi(x_1)-\tilde{f}\star\phi(x_2)}_\infty \leq Lip(f)d_0(x_1,x_2).
		 \end{split}
		 \]
	 Hence, \(f_*\) is Lipschtiz with \(Lip(f_*)=Lip(f)\). Now, using Lemma~5.2 in \cite{CL}, we have
	     \[
		 \begin{split}
		 \norm{f_*}_{\cC^r}&=\norm{gf+(1-g)\tilde{f}\star \phi}_{\cC^r}\\
         &\leq\norm{f}_{\cC^r}\norm{g}_{\cC^r}+ (1-\norm{g}_{\cC^r})\norm{\int\tilde{f}(x-y)\phi(y)dy}_{\cC^r}\\
		 &\leq\norm{f}_{\cC^r}\norm{g}_{\cC^r}+ (1-\norm{g}_{\cC^r})\int\norm{f}_{\cC^r}\norm{\phi(y)}_{\cC^r}dy\\
		 &< \infty.
		 \end{split}
		 \]
	 Therefore, \(f_*\in \cC^r\).
\end{proof}

\section{Transversality}\label{sec:transversality}
For the reader's convenience, we state the transversality theorem as used in the main text. We refer to \cite{Arnold88} for details. The Theorem in \cite[Chapter 6, section 29.E]{Arnold88} is stated for smooth maps and manifolds, but it can easily be reduced to the following version just using the \(\cC^r\) version of Sard's theorem. Also,  the author discusses in detail the extension of the theorem to stratified subvarieties, which is the case we are interested in and for which we state the theorem. One can also find the \(\cC^r\) version of the Transversality theorem in \cite{Abraham63}, but there, the reader needs to be mindful of the specific properties they ask on manifolds.
\begin{defin}[Transversal mapping] \label{def:transversal}
For every manifolds $A,B$ and submanifold $C\subset B$,   a \(\cC^r\) mapping \(f:A\to B\) is said to be transversal to \(C\) \((f \pitchfork C)\) at a point \(a\) if either \(f(a)\notin C\) or the tangent plane to \(C\) at \(f(a)\), if $f(a)\in\partial C$, and the image of the tangent plane to \(A\) at \(a\) are transversal:
         \[
		 D_a f(T_a A)\oplus T_{f(a)}C=T_{f(a)}B.
		 \]
     \(f\) is said to be transversal to \(C\) if \(f\) is transversal to \(C\) at \(a\) for every \(a\in A\).
\end{defin}
\begin{defin} 
For every manifolds $A,B$ and stratified subvariety $C\subset B$,   a \(\cC^r\) mapping \(f:A\to B\) is said to be transversal to \(C\) if it is transversal to $C$ and all its substrata.
\end{defin}
\begin{thm}[Transversality Theorem-\(\cC^r\) version.\!\! {\cite[Chaper 6, section 29.E]{Arnold88}}]\label{thm:tt}
	 Let \(A\) be a compact manifold, and let \(C\) be a compact stratified subvariety of a manifold \(B\), then the \(\cC^r\) mappings \(f:A\to B\) with \(f \pitchfork C\) form an open everywhere dense set in the space of all \(\cC^r\) mappings \(A\to B\). 
\end{thm}
We will apply the above Theorem to the following situation ( \(Y\) and the functions \(\phi_i\) are as defined in the Section~\ref{sec:ITF}). 
\begin{lem} \label{lem:intersection}
Let \(Y=\{x:\norm{x}\leq R\}\subset\bR^d\), for some \(R>0\), $B\subset Y$ be open and $\phi_i\in\cC^3(\bR^d,\bR^d)$, $i\in\{1,\dots,m\}$ be such that they are invertible when restricted to $B$. Also, let $W_i\subset B$, $i\in\{1,\dots, m\}$, be $d_i$-dimensional compact manifolds with boundaries.
The maps $F:\bR^d\to\bR^{md}$, defined by $F(x):=(\phi_{1}(x),\dots, \phi_{m}(x))$,  which are transversal to the stratified subvariety $C=W_1\times\cdots\times W_m$ form an open and dense set. Moreover, if \(F|_{B} \pitchfork C\) then the manifold $\left(\bigcup_{i=1}^m \phi_i^{-1} (W_i\setminus \partial W_i)\right)\bigcap B$ is empty if there exists $k$ such that $\sum_{i=1}^md_i-(k-1)d<0$, otherwise it has dimension at most $\sum_{i=1}^md_i-(m-1)d$.
\end{lem}
\begin{proof}
 Our strategy is to transform the present setting to the setting suitable for the application of Theorem~\ref{thm:tt} and then proceed further.\\
 The fact that $W_1\times\cdots\times W_k$ is a stratified subvariety can be checked directly.
A minor problem is that neither \(Y\) nor \(\bR^d\) are compact manifolds. To overcome this problem, define the function $g\in \cC^\infty(\bR,\bR)$ such that $g(y)=1$ for $y\leq 1$ and $g(y)=0$ for $y\geq a$, with $|g’(x)|\leq 2a^{-1}$ (\(a>2\) to be chosen large enough) and define \(\hat F:\bR^d\to \bR^{md}\) as
\[
\hat F(x)= g(R^{-1}\|x\|)F(x)
\]
Note that for  $x\in Y$ $\hat F=F$, while, for \(\norm{x}>aR\), \(\hat F(x)=0\). Hence \(\hat F\) can be seen as a smooth function on the torus $\bR^m/ \bZ_{aR}$ (which indeed is a compact manifold). We can thus apply Theorem \ref{thm:tt} to obtain the first part of the Lemma.\\
To prove the second part, note that if $z\in B$ is such that $\hat F(y)=F(y)\in C\setminus \partial C$ we have
 \begin{equation}\label{eq:transverse} 
 D_y F(\bR^d)+T_{F(y)}C=\bR^{md}.
 \end{equation}
	First of all note that if there exists a \(y\in \bR^d\) such that \(F(y)\in C\setminus \partial C\) then $\phi_i(y)\in W_i$, hence $y=\phi_i^{-1}(W_i)$, that is $\bigcap_{i=1}^m \phi_i^{-1}(W_i)\supset \{y\}\neq \emptyset$.
Next, let $\bar d=\sum_{i=1}^m d_i$, if $d+\bar d< md$, then \eqref{eq:transverse} cannot be satisfied. It follows that if $\bar d<(m-1)d$, then $\bigcap_{i=1}^m\phi_i^{-1}(W_i)= \emptyset$.\\
If $\bar d\geq(m-1)d$, we study equation \eqref{eq:transverse}: for an arbitrary $(\tilde \beta_1,\dots, \tilde \beta_m)\in\bR^{md}$ there must exists $\alpha\in \bR^d$ and $\tilde w_i\in TW_i$ such that $D\phi_i\alpha+\tilde w_i=\tilde \beta_i$. So setting $\beta_i=D\phi_i^{-1}(\tilde\beta_i)$ and $w_i=D\phi_i^{-1}( \tilde w_i)$ we must study the solutions of
\begin{equation}\label{eq:transverse1}
\alpha+w_i=\beta_i.
 \end{equation}
Note that $\alpha$ is uniquely determined by $\alpha=\beta_1-w_1$. Subtracting the second of the \eqref{eq:transverse1} from the first yields $w_1-w_2=\beta_1-\beta_2$. If $d_1+d_2<d$, such an equation has no solution for all $\beta_2$, so the intersection must be empty. If $s_2=d_1+d_2-d\geq 0$, then the dimension of $W_{1,2}:=D\phi_i^{-1}W_1\cap D\phi_i^{-1}W_2$ is $s_2$. We can then write $w_1=\xi_1+\hat w_1$ and $w_2=\xi_1+\hat w_2$ with $\hat w_i\in D\phi_i^{-1}W_i\cap W_{1,2}^\perp$. It follows $\hat w_1-\hat w_2=\beta_1-\beta_2$ which determines uniquely $\hat w_1,\hat w_2$.
We can then write $w_3-\xi_1=\beta_3-\beta_2+\hat w_2$. Let $s_3=d_3+s_2-d=d_1+d_2+d_3-2d$, if $s_3<0$ again in general there are no solutions. Otherwise the dimension of $W_{1,2,3}=W_{1,2}\cap  D\phi_i^{-1}W_3$ is $s_3$ and we can write $\xi_1=\xi_2+\hat \xi_1$, $w_3=\xi_2+\hat w_3$ with $\xi_2\in W_{1,2,3}$ and $\hat\xi\in W_{1,2}\cap W_{1,2,3}^\perp$, $\hat w_3= D\phi_i^{-1}W_3\cap W_{1,2,3}^\perp$.
Accordingly, we have $\hat w_3-\hat \xi_1=\beta_3-\beta_2+\hat w_2$, which determines uniquely $\hat w_3,\hat \xi_1$. Continuing in such a way, we have that $W_{1,\dots,m}=D\phi_i^{-1}W_1\cap \cdots\cap D\phi_i^{-1}W_m$ has dimension $\bar d-(m-1) d$.
The case in which $F(y)$ belongs to a substrata of $C$ is treated in exactly the same way and yields a lower dimension.
\end{proof}

\section{Technical Lemmata}\label{sec:technical}
In this section, for the reader's convenience, we collect some simple but boring technical results.
\begin{lem}\label{lem:good_pert}
There exists $\Const>0$ such that, for given \(\ve\in (0,1/2)\), and \(f,h\in\cC^3\), where \(h\) is a diffeomorphism such that \(\norm{h-id}_{\cC^2}<\ve\), we have  \(h\circ f\circ h^{-1}\in \cC^3\) and 
	 \[
	 \begin{split}
	& \norm{h^{-1}-id}_{\cC^2}\leq 6\ve\\
	 &\norm{h\circ f\circ h^{-1}-f}_{\cC^2}<\Const \|f\|_{\cC^3} \ve.
	\end{split}
	\]
\end{lem}
\begin{proof}
	 First we claim that \(\norm{h-id}_{\cC^2}<\ve\) implies \(\norm{h^{-1}-id}_{\cC^1}<2\ve\). 
	 Indeed there there exists a transformation \(A\), with \(\norm{A}_{\cC^1}<1\), such that we can write \(Dh=\Id+\ve A\) therefore
	     \[
		 (Dh)^{-1}=(\Id+A)^{-1}=\sum_{k=0}^\infty(-1)^k\ve^kA^k.
	     \]
	 That is, 
	     \[
	     \begin{split}
		 &\norm{(Dh)^{-1}-\Id}_{\cC^0}\leq \sum_{k=1}^\infty \ve^k\|A\|_{\cC^0}^k\leq\frac \ve{1-\ve}\leq 2\ve\\
		& \norm{D(Dh)^{-1}}_{\cC^0}\leq \sum_{k=1}^\infty \ve^k\|A\|_{\cC^1}^k\leq 2\ve.
	    \end{split}
	     \]
Moreover, the inverse function theorem implies $h^{-1}\in\cC^3$. Accordingly,  since \(Dh^{-1}=(Dh)^{-1}\circ h^{-1}\), we have
	     \begin{align*}
		 \norm{Dh^{-1}-\Id}_{\cC^0}&=\norm{((Dh)^{-1}-h)\circ h^{-1}}_{\cC^0}= \norm{((Dh)^{-1}-\Id)\circ h^{-1}}_{\cC^0}\\
		 &=\norm{(Dh)^{-1}-\Id}_{\cC^0}<2\ve\\
		  \norm{Dh^{-1}-\Id}_{\cC^1}&\leq2\ve+\norm{D(Dh)^{-1}}_{\cC^0}\norm{Dh^{-1}}_{\cC^0}\leq 6\ve.
		 \end{align*}
Hence,
	     \[
\begin{split}
        & \bigl\|h\circ f\circ h^{-1}-f\bigr\|_{\cC^0}\leq \ve+\|f\|_{\cC^1}\ve\\
	 &\bigl\|D(h\circ f\circ h^{-1}-f)\bigr\|_{\cC^0}=\bigl\|Dh\circ f\circ h^{-1}\cdot Df\circ h^{-1}\cdot Dh^{-1}-Df)\bigr\|_{\cC^0}\\
	 &\phantom{\bigl\|D(h\circ f\circ h^{-1}-f)\bigr\|_{\cC^0}=}
	 \leq \bigl\|Df\circ h^{-1}-Df\bigr\|_{\cC^0}+ \Const \ve  \norm{f}_{\cC^1}\leq \Const \norm{f}_{\cC^2}\ve.
\end{split}
		 \]
Next,     
\[
\begin{split}
	 \bigl\| \partial_xD(h\circ f\circ h^{-1}-f)\bigr\|_{\cC^0} & = \bigl\|( \partial_yDh)\circ f\circ h^{-1} (\partial_x f)_y\circ h^{-1} \partial_x (h^{-1}_z) D(f\circ h^{-1})\\
	 &\phantom{ = }  + Dh\circ f\circ h^{-1}(\partial_yD f)\circ h^{-1} \partial_x(h^{-1}_y) Dh^{-1}\\
	 &\phantom{ = }  +Dh\circ f\circ h^{-1}Df\circ h^{-1}\partial_x(Dh^{-1})-\partial_x Df\bigr\|_{\cC^0}\\
	 &\leq \bigl\| (\partial_xD f)\circ h^{-1}-\partial_x Df\bigr\|_{\cC^0}+ \Const\|f\|_{\cC^2}\ve \\
	 &\leq \Const\|f\|_{\cC^3}\ve.
\end{split}
		 \]
Thus
	     \[
\begin{split}
	 \norm{h\circ f\circ h^{-1}-f}_{\cC^2}=& \norm{h\circ f\circ h^{-1}-f}_{\cC^0}+\norm{D(h\circ f\circ h^{-1}-f)}_{\cC^0}\\
	 &+\norm{D^2(h\circ f\circ h^{-1}-f)}_{\cC^0}\leq \Const\|f\|_{\cC^3}\ve.
\end{split}
		 \]
	 for some constant \(\Const>0\). Finally  \(h\circ f\circ h^{-1}\in \cC^3\) being the composition of $\cC^3$ function.
\end{proof}
\begin{lem}\label{lem:good_pert2}
	 For given \(\ve>0\), let \(g, h\in\cC^3\) be such that \(g\) is invertible, \(\norm{g-h}_{\cC^2}<\ve\) then \(g^{-1}\circ h\in\cC^3\) and \
	     \[
	     \norm{g^{-1}\circ h-id}_{\cC^2}\leq C_\#\ve\|g^{-1}\|^3_{\cC^3}.
	     \] 
	 Additionally for \(f\in\cC^3\), \(f\circ g^{-1}\circ h, (f\circ g^{-1}\circ h)^{-1}\in \cC^3\) and 
	     \[
			\norm{f\circ g^{-1}\circ h-id}_{\cC^2}\leq C_\#\norm{f}_{\cC^3}\ve\|g^{-1}\|^3_{\cC^3}
		 \]
\end{lem}
\begin{proof} 
Since \(g\in\cC^3\) is invertible, by inverse function theorem \(g^{-1}\in\cC^3\) and therefore being a composition of $\cC^3$ functions \( g^{-1}\circ h\in \cC^3\).
Next, let $\Psi=g^{-1}\circ h$. Then
\[
\begin{split}
&\|\Psi-id\|_{\cC^0}=\|g^{-1}\circ h-g^{-1}\circ g\|_{\cC^0}\leq \|g^{-1}\|_{\cC^1}\ve\\
&\|D(\Psi-id)\|_{\cC^0}=\|(Dg)^{-1}\circ \Psi\cdot Dh-\Id\|_{\cC^0}\leq \|g^{-1}\|_{\cC^1}\|g^{-1}\|_{\cC^2}\ve\\
&\|\partial_x D(\Psi-id)\|_{\cC^0}=\|\partial_y[(Dg)^{-1}]\circ \Psi\partial_x \Psi_y\cdot Dh+(Dg)^{-1}\circ \Psi\cdot\partial_x Dh\|_{\cC^0}\\
&\leq \|\partial_y[(Dg)^{-1}]\circ \Psi\partial_x \Psi_y \cdot Dg+(Dg)^{-1}\circ \Psi\cdot\partial_x Dg\|_{\cC^0}+\Const\ve \|g^{-1}\|_{\cC^2}^2\\
&\leq \|\partial_x[(Dg)^{-1}\circ \Psi \cdot Dg]\|_{\cC^0}+\Const\ve \|g^{-1}\|_{\cC^2}^2\\
&\leq \|\partial_x[(Dg)^{-1} \cdot Dg]\|_{\cC^0}+\Const\ve \|(Dg)^{-1}\|_{\cC^3}^3=\Const\ve \|g^{-1}\|_{\cC^3}^3.
\end{split}
\]
Accordingly, $\|\Psi-id\|_{\cC^2}\leq \Const\ve \|(Dg)^{-1}\|_{\cC^3}$. Moreover, the first part of Lemma \ref{lem:good_pert} implies that $\Psi$ is invertible and $\|\Psi^{-1}-id\|_{\cC^2}\leq \Const\ve \|(Dg)^{-1}\|_{\cC^3}$.
This implies, $f\circ \Psi, \Psi^{-1}\circ f^{-1}\in\cC^3$ and
\[
\begin{split}
&\left\|f\circ \Psi-f\right\|_{\cC^0}\leq  \Const\ve\|f\|_{\cC^1} \|g^{-1}\|_{\cC^1}\\
&\left\| D(f\circ \Psi-f)\right\|_{\cC^0}=\bigl\|(Df)\circ \Psi\cdot D\Psi-Df \bigr\|_{\cC^0}\le\Const \ve\|f\|_{\cC^2}\|g^{-1}\|_{\cC^2}^2\\
&\left\| \partial_x D(f\circ \Psi-f)\right\|_{\cC^0}=\left\|[\partial_y(Df)]\circ \Psi\cdot \partial_x\Psi_y\cdot D\Psi+(Df)\circ \Psi\cdot\partial_x D\Psi -\partial_xDf \right\|_{\cC^0}\\
&\phantom{\left\| \partial_x D(f\circ \Psi-f)\right\|_{\cC^0}}
\leq \Const \|f\|_{\cC^3}\|g^{-1}\|_{\cC^3}^3,
\end{split}
\]
from which the Lemma follows.
\end{proof}

Let \(f\) be a piecewise smooth contraction with an associated IFS \(\Phi=\{\phi_1,\ldots,\phi_m\}\) and as in equation~(\ref{eq:admissible}), \(\Sigma^m_{n,i}(\Phi)\) be the set of \(i\)-admissible sequences and as in equation~(\ref{eq:bdnbd}), \(D^{N}_\delta(\Phi)\) be a \(\delta-\)neighbourhood of the boundary of partition \(\bsP(f^{N})\) for \(\delta>0\). 
We have the following result:
\begin{lem}\label{lem:good_diff}
	 For a piecewise smooth contraction \(f\) with IFS \(\Phi=\{\phi_1,\ldots,\phi_m\}\) and \(N\in\bN\), there exists \(\ve>0\) such that for \(\tilde f\) with associated IFS \(\tilde \Phi=\{\tilde\phi_1,\ldots,\tilde\phi_m\}\)  satisfying \(d_2(f,\tilde f)<\ve\), we have \(\Sigma^m_{n,i}(\Phi)=\Sigma^m_{n,i}(\tilde\Phi)\). Moreover, there exists \(\delta>0\) such that \(D^{N}_{\delta/4}(\tilde\Phi)\subset D^{N}_{\delta/2}(\Phi)\).
\end{lem}
\begin{proof} By hypothesis there exists $M>0$ such that $\|\phi_i|_{U_i}^{-1}\|_{\cC^1}\leq M$ and we can restrict to such set by the definition of $D^{N}$ which entails only admissible sequences. Thus
\begin{align*}
	 \norm{\tilde\phi_i^{-1}-\phi_i^{-1}}_{\cC^0}&=\norm{id-\phi_i^{-1}\circ\tilde\phi_i}_{\cC^0}=\norm{\phi_i^{-1}\circ\phi_i-\phi_i^{-1}\circ\tilde\phi_i}_{\cC^0}\\
	 &\leq \norm{\phi_i^{-1}}_{\cC^1} \norm{\phi_i-\tilde\phi_i}_{\cC^0}\leq M\norm{\phi_i-\tilde\phi_i}_{\cC^0}.
\end{align*}
For \(n\leq N\) and  admissible sequence \(\sigma=(\sigma_1,\sigma_2,\ldots,\sigma_n)\in\Sigma^m_n\), we have 
\[
\begin{split}
     &\norm{\tilde\phi_{\sigma_n}^{-1}\circ\tilde\phi_{\sigma_{n-1}}^{-1}\circ\cdots\circ \tilde\phi_{\sigma_1}^{-1}-\phi_{\sigma_n}^{-1}\circ\phi_{\sigma_{n-1}}^{-1}\circ\cdots\circ\phi_{\sigma_1}^{-1}}_{\cC^0}=\\
     & =\norm{id-\phi_{\sigma_n}^{-1}\circ\cdots\circ\phi_{\sigma_1}^{-1}\circ\tilde\phi_{\sigma_1}\circ\cdots\circ\tilde\phi_{\sigma_n}}_{\cC^0}\\
	 &=\norm{\phi_{\sigma_n}^{-1}\circ\cdots\circ\phi_{\sigma_1}^{-1}\circ\phi_{\sigma_1}\circ\cdots\circ\phi_{\sigma_n} -\phi_{\sigma_n}^{-1}\circ\cdots\circ\phi_{\sigma_1}^{-1}\circ\tilde\phi_{\sigma_1}\circ\cdots\circ\tilde\phi_{\sigma_n}}_{\cC^0}\\
	 &\leq M^n\sum_{i=1}^{n}\norm{\phi_{\sigma_1}\circ\cdots\circ\phi_{\sigma_{i}}\circ \tilde\phi_{\sigma_{i+1}}\circ\cdots\circ\tilde\phi_{\sigma_n}-\phi_{\sigma_1}\circ\cdots\circ \phi_{\sigma_{i-1}}\circ \tilde\phi_{\sigma_{i}}\circ\cdots\circ\tilde\phi_{\sigma_n}}_{\cC^0}\\
	 &\leq M^n(1-\lambda)^{-1}\sup_i\norm{\phi_i-\tilde\phi_i}_{\cC^0}.
\end{split}
\]
Note that, by definition \eqref{eq:admissible}, there exists $\ve_0>0$ such that, for each $\ve$-perturbation $\tilde\Phi$ of $\Phi$, with $\ve\in (0,\ve_0)$, $\Sigma^m_{n,i}(\Phi)=\Sigma^m_{n,i}(\tilde \Phi)$. Moreover, for each $\sigma \in \Sigma^m_{n,i}(\Phi)$ and $\xi\in \psi_i^{-1}\left(M_{\sigma,i}(\Phi)\right)$,
\[
\|\phi_{\sigma_n}^{-1}\circ\cdots\circ\phi_{\sigma_1}^{-1}\circ\psi_i(\xi)-\tilde \phi_{\sigma_n}^{-1}\circ\cdots\circ\tilde \phi_{\sigma_1}^{-1}\circ\psi_i(\xi)\|\leq  M^n(1-\lambda)^{-1}\ve
\]
Thus, for $\ve$ small enough $D^{N}(\tilde \Phi)\subset D^{N}_{\delta/4}(\Phi)$, hence $D^{N}_{\delta/4}(\tilde \Phi)\subset D^{N}_{\delta/2}(\Phi)$.
\end{proof}

\end{document}